\documentclass[onefignum,onetabnum]{article}
\usepackage[utf8]{inputenc}
\usepackage{amsmath}
\usepackage{amssymb}
\usepackage{xcolor}
\usepackage{epsfig} 
\usepackage{soul}

\usepackage{algorithmic}

\usepackage[raggedright]{titlesec} 
\usepackage{cleveref} 
\usepackage{fancyhdr} 
\usepackage{amsthm} 
\usepackage{algorithm2e} 
\usepackage{url} 
\newtheorem{lemma}{Lemma}
\newtheorem{theorem}{Theorem}
\usepackage{authblk}

\newcommand{\eps}{\varepsilon}

\newtheorem{remark}{Remark}

\renewcommand{\epsilon}{\varepsilon}

\title{Slow passage through a Hopf-like bifurcation in piecewise linear systems: application to elliptic bursting}

\author[1]{J. Penalva\thanks{j.penalva@uib.cat}}
\author[2]{M. Desroches\thanks{mathieu.desroches@inria.fr}}
\author[1]{A. E. Teruel\thanks{antonioe.teruel@uib.es}}
\author[1]{C. Vich\thanks{catalina.vich@uib.es}}
\affil[1]{Departament de Matem\`atiques i Inform\`atica \& IAC3, Universitat de les Illes Balears, Spain}
\affil[2]{MathNeuro Team, Inria Sophia Antipolis M\'editerran\'ee Research Centre, France}

\usepackage{enumitem}
\setlist[itemize,enumerate,1]{leftmargin=\dimexpr 26pt}

\begin{document}

\maketitle

\begin{abstract}
The phenomenon of slow passage through a Hopf bifurcation is ubiquitous in multiple-timescale dynamical systems, where a slowly-varying quantity replacing a static parameter induces the solutions of the resulting slow-fast system to feel the effect of a Hopf bifurcation with a delay. This phenomenon is well understood in the context of smooth slow-fast dynamical systems. In the present work, we study for the first time this phenomenon in piecewise linear (PWL) slow-fast systems. This special class of systems is indeed known to reproduce all features of their smooth counterpart while being more amenable to quantitative analysis and offering some level of simplification, in particular through the existence of canonical (linear) slow manifolds. We provide conditions for a PWL slow-fast system to exhibit a slow passage through a Hopf-like bifurcation, in link with the number of linearity zones considered in the system and possible connections between canonical attracting and repelling slow manifolds. In doing so, we fully describe the so-called way-in/way-out function. Finally, we investigate this slow passage effect in the Doi-Kumagai model, a neuronal PWL model exhibiting elliptic bursting oscillations.
\end{abstract}



\section{Introduction}\label{sec:1intro}

Differential systems involving multiple timescales are used in a wide range of applications to describe the dynamical behavior of a given set of variables. Let us consider a system of the form
\begin{equation}\label{eq:intro_sys}
\begin{split}
\dot{\mathbf{x}}&=F(\mathbf{x},z),\\
\dot{z}&=\eps,
\end{split}
\end{equation}
where $\mathbf{x}\in\mathbb{R}^q$, ${z}\in\mathbb{R},$ $F$ is a sufficient smooth function, and $0 \leq \eps\ll 1$ is a small parameter. The presence of $\eps$ in~\eqref{eq:intro_sys} induces a timescale separation such that the variables $\mathbf{x}$ are said to be \textit{fast} while $z$ is \textit{slow}. If we consider $\eps=0$ in~\eqref{eq:intro_sys}, which yields the so-called \textit{fast subsystem} or \textit{layer problem}, then $z$ becomes a parameter for the differential equation $\dot{\mathbf{x}}=F(\mathbf{x},z)$. An important object is the so-called \textit{critical manifold} $\mathcal{S}_0=\{(\mathbf{x},z):F(\mathbf{x},z)=\mathbf{0}\}$, which corresponds to the $z$-dependent family of fast subsystem equilibria associated with system~\eqref{eq:intro_sys}. In the present work, we consider cases where $\mathcal{S}_0$ possesses a Hopf bifurcation at a certain value $z=z_H$ where the fast subsystem equilibrium loses its stability. Without loss of generality, we further assume that the critical manifold $\mathcal{S}_0$ is a graph over $z$, for $z$ within an interval $I$ that contains $z_H$, hence we can write $\mathcal{S}_0=\{(S_0(z),z):z\in I\}$.

In the full system~\eqref{eq:intro_sys}, $z$ drifts slowly in time, hence there is no Hopf bifurcation anymore. Instead, the system undergoes a \textit{dynamic} (or \textit{delayed}) \textit{Hopf bifurcation} at $z=z_H$, also referred to as a {\it slow passage} through the Hopf bifurcation at $z_H$. One may expect that orbits of system~\eqref{eq:intro_sys} with initial conditions $(\mathbf{x}_0,z_0)$, where $z_0<z_H$ and $\mathbf{x}_0$ close to $S_0(z_0)$, stay close to the critical manifold $\mathcal{S}_0$ for $0<t<\frac{z_H-z_0}{\eps}:=t_H$ and  move away from it as soon as $t>t_H$. Instead, a surprising delay phenomenon occurs, where orbits remain close to $\mathcal{S}_0$ for a long period of time $\tau^*>t_H$, that is, past the fast subsystem bifurcation point, and then they move away from $\mathcal{S}_0$ only after $z^*=z(\tau^*)$ \cite{neishtadt1987persistence,neishtadt1988persistence}. This value $ z^*$ (and the corresponding time $\tau^*$) marks therefore a \textit{delay} in the response of the full system~\eqref{eq:intro_sys} to the fast subsystem instability occurring at $z_H$. This delay depends on $\eps$ and is expected to tend to the bifurcation value $z_H$ as $\eps$ tends to zero, given that, when $ \eps = 0 $, the change of stability occurs exactly at $z_H$.
Nevertheless, different limiting behaviors can be obtained for $z^*(\eps)$, mainly depending on the regularity of the vector field; see \cite{neishtadt2009stability} for details. In particular there can be: no delay; a fixed \textit{maximal} delay in the interval $(z_H,+\infty)$, which is then called a \textit{buffer point}; or the delay may be arbitrarily large.

The phenomenon of slow passage through a Hopf bifurcation (delayed loss of stability) occurs in models coming from a vast range of application areas; for instance, in Neuroscience~\cite{Baer1989}, Physics \cite{Premraj2016}, Population Biology~\cite{kuznetsov1996remarks} or Chemistry~\cite{Holden1993}, to name but a few. Slow-fast dynamical systems in which a delayed loss of stability happens often display complex oscillatory behaviors such as bursting~\cite{HE93-2,HMS85,Rinzel1987,Coombes2005} or mixed-mode oscillations (MMOs)~\cite{Baer1989,diener1991maximal,HB12,Huagan}. In this context, delayed Hopf bifurcation can be related to the \textit{canard} phenomenon~\cite{callot1981chasse,desroches2012,diener1991maximal,10.1007/BFb0085020}.

It is well known that continuous piecewise linear (PWL) systems provide a simple and minimal framework to study nonlinear dynamics. More specifically, one can reproduce a wide range of multiple-timescale dynamics, in particular in the canard regime, with PWL slow-fast systems~\cite{CFT19,desroches:hal-01651907,desroches2013,desroches2016,fernandez2016canard,Prohens2013,ProhensVich}. Over the past few years, we have indeed shown that such systems are able to retain all salient features of their smooth counterpart while allowing for substantial simplification: canonical linear slow manifolds, explicit expressions for solutions and access to quantitative information related to various flight times along solutions.

In the present work, we analyze the slow passage through a Hopf bifurcation in the context of PWL systems, where this bifurcation is known as Hopf-like bifurcation. We perform this analysis using a series of minimal PWL models presenting delayed loss of stability near a fast subsystem Hopf-like bifurcation, and we study the limiting behavior of their respective delay as the singular parameter $\varepsilon$ tends to zero. We then apply our results to a PWL model of elliptic bursting oscillations as this type of bursting relies upon a slow passage through a Hopf bifurcation~\cite{izhikevich2000neural,Su_2003,Izhikevich2000}.

The rest of the manuscript is organized as follows. In~\cref{sec.2regimes} we analyze a minimal PWL system with two linear regions exhibiting delayed loss of stability due to a slow passage through a Hopf-like bifurcation, such that the delay tends to the bifurcation value $z_H$ as $\varepsilon$ tends to zero. By adding a third region allowing for the connection between the attracting and the repelling canonical slow manifolds, we build up in~\cref{sec.3regimes} a minimal PWL model also exhibiting delayed loss of stability, but with an arbitrarily large delay. Then, by adding an equilibrium point on the repelling canonical slow manifold, we analyze the existence of a buffer point and the associated maximal delay. Finally, in \cref{sec:4dk} we apply locally all the previous results to the Doi-Kumagai PWL elliptic bursting model, which we revisit by considering first two and then three linearity zones near the fast subsystem Hopf-like bifurcation point, respectively.
\section{Two-regions system}\label{sec.2regimes}
Let us consider system 
\begin{align}\label{sys:2pieces}
 \dot{x}&=f(x)-y,\nonumber\\
 \dot{y}&=x-z,\\
 \dot{z}&=\varepsilon, \nonumber
\end{align}
where $\varepsilon$ is a small parameter, i.e. $0\leq \varepsilon\ll 1$, and $f$ is the continuous piecewise linear function
\begin{equation}\label{eq:x-nullcina1}
f(x)=\left\{
\begin{array}{rl}
-mx & x \leq  0,\\
kx & x \geq 0,
\end{array}
\right.
\end{equation}
with  $0<k,m<2$. It is therefore a 2-fast 1-slow PWL system of the form given by~\eqref{eq:intro_sys}, whose critical manifold $\mathcal{S}_0$ is the polygonal line 
\begin{equation}\label{eq:crtmfold2pieces}
 \mathcal{S}_0=\left\{
    \begin{array}{ll}
      (x,-mx,x) & \text{if }x\leq  0,\\
      (x,kx,x)  & \text{if }x> 0.
    \end{array}
 \right.
\end{equation}

When $\eps=0$, the fast subsystem of~\eqref{sys:2pieces} 
corresponds to the planar PWL system 
\begin{equation} \label{eq:odeInit}
\begin{array}{rcl}
\dot{x} & = & f(x)-y,\\
\dot{y} & = & x-z,
\end{array}
\end{equation}   
with $z$ acting as a parameter. This system exhibits a unique equilibrium point at $(z,f(z))$ and, when $z\neq 0$, the Jacobian matrix evaluated at the equilibrium point is
\[
 \begin{pmatrix}
  f'(z) & -1\\
  1	& 0
 \end{pmatrix},
\]
where $f'(z)$ is the derivative of the PWL function $f$ at $z$. This equilibrium point changes stability at $z=0$, passing from a stable focus when $z<0$ to an unstable focus when $z>0$. Following Theorem 5 in~\cite{FPT99} the system can exhibit either a supercritical or a subcritical Hopf-like bifurcation at $z=0$, where a family of stable (resp. unstable) limit cycles emerges and exists for all $z>0$ (resp., for all $z<0$). The amplitude of the limit cycles along both these families grows linearly with $\lvert z\rvert$; see also~\cite{simpson2018compendium}.

In order to study the slow passage phenomenon through a Hopf-like bifurcation, we consider the full system~\eqref{sys:2pieces}, in which $z$ is a slow variable, driving the dynamics of the fast subsystem~\eqref{eq:odeInit} through the Hopf-like bifurcation at $z=0$. In fact,  from~\eqref{sys:2pieces}, we only consider $z$ as a slow drift.

In the next result we describe the invariant manifolds of system~\eqref{sys:2pieces}, both before and after perturbation in $\eps$. In particular, we study the normally hyperbolic branches of the critical manifold. We recall that normally hyperbolic manifolds are those where the normal component of the flow dominates the tangential component of the flow.

\begin{lemma}\label{lema:inv.sets.2r}
Let us consider system  \eqref{sys:2pieces}-\eqref{eq:x-nullcina1} with $m\ne k$. For $\eps =0$, the system exhibits two  invariant sets. One invariant set is the critical manifold $\mathcal{S}_0$, which has a normally hyperbolic attracting branch, $S_0^a$, defined for $x<0$ and a normally hyperbolic repelling one, $S_0^r$, defined for $x> 0$. Another invariant set is a cone, $\mathcal{C}$,  with vertex at the origin and foliated by periodic orbits, which is stable when $k<m$ and unstable when $k>m$ (see~\cref{fig:invsets}(a)).

For small enough $\eps>0$, the attracting branch of the critical manifold $\mathcal{S}^a_0$ perturbs to a canonical attracting slow manifold $\mathcal{S}_{\eps}^a$, and the repelling branch of the critical manifold $\mathcal{S}_{0}^r$ perturbs to a canonical repelling slow manifold $\mathcal{S}_{\eps}^r$ where
\begin{equation} \label{eq:S.eps}
    \begin{array}{ll}
    \mathcal{S}_{\eps}^a =  (x,-mx-\eps,x+m\eps) & \text{if }x \leq 0, \\
    \mathcal{S}_{\eps}^r = (x,kx-\eps,x-k\eps) & \text{if }x \geq 0. 
    \end{array}
\end{equation}
These manifolds intersect the plane $\{x=0\}$ at points $\mathbf{p}^a=(0,-\eps,m\eps)$ and $\mathbf{p}^r=(0,-\eps,-k\eps)$, respectively (see~\cref{fig:invsets}(b)).
\end{lemma}
\begin{proof}
System \eqref{sys:2pieces}-\eqref{eq:x-nullcina1} can be written as the 3-dimensional PWL system
$$
\dot{\mathbf{u}}=
  \left\{
	\begin{array}{ll}
	  A_{-}\mathbf{u}+\varepsilon\,\mathbf{e}_3 & \text{if } x \leq 0, \\
	  A_{+}\mathbf{u}+\varepsilon\, \mathbf{e}_3 & \text{if }x \geq 0,
	\end{array}
  \right.
$$
where $\mathbf{u}=(x,y,z)^T$, $\mathbf{e}_3=(0,0,1)^T$ and
\[
 A_-=\begin{pmatrix}
      -m & -1 & 0 \\
      1 & 0 & -1\\
      0 & 0 & 0
     \end{pmatrix},\quad
 A_+=\begin{pmatrix}
      k & -1 & 0 \\
      1 & 0 & -1\\
      0 & 0 & 0
     \end{pmatrix}.     
\]

The critical manifold $\mathcal{S}_0$ decomposes as $\mathcal{S}_0^{a} \cup \{0\}\cup \mathcal{S}_0^{r} $, $\mathcal{S}_0^{a}$ being the subset contained in the half-space $\{x<0 \}$ and $\mathcal{S}_0^{r}$ 
the subset contained in $\{x>0\}$.

The stability of the branches $\mathcal{S}_0^a$ and $\mathcal{S}_0^r$ depends on the sign of the real part of the non-null eigenvalues of $A_{-}$ and $A_{+}$. These eigenvalues are $-\frac{m}2\pm\frac {\sqrt{4-m^2}}2 \rm{i}$  and $\frac{k}2\pm\frac {\sqrt{4-k^2}}2 \rm{i}$, respectively. The real part of the first ones is negative hence $\mathcal{S}_0^a$ is a normally hyperbolic attracting branch, while the real part of the second ones is positive and hence $\mathcal{S}_0^r$ is a normally hyperbolic repelling branch. Notice that the origin is an equilibrium point located at the switching plane, where the Jacobian matrix is not defined. Therefore, the origin is not a normally hyperbolic point.

On the other hand, the existence of the stable (resp. unstable) cone $\mathcal{C}$ is a consequence of the supercritical (resp. subcritical) Hopf-like bifurcation exhibited by the fast subsystem  \eqref{eq:odeInit} when $0<k<m<2$ (resp. $0<m<k<2$) at $z=0$; see Theorem 5(d) in \cite{FPT99}, where we consider  $\gamma_L=-\frac{m}{\sqrt{4-m^2}}$, $\gamma_R=\frac{k}{\sqrt{4-k^2}}$. We point out that in the Hopf-like bifurcation, the amplitude of the limit cycle grows linearly with  $\lvert z\rvert$, which guarantees the conic shape.

When $\varepsilon>0$, both rays in expression \eqref{eq:S.eps} are invariant under the flow of the system  \eqref{sys:2pieces}-\eqref{eq:x-nullcina1}.  Moreover, each of these rays is at distance of order $\varepsilon$ to the respective branch of the critical manifold $\mathcal{S}_{0}$. Following \cite{ProhensVich} we conclude that $\mathcal{S}^a_{\varepsilon}$ is a canonical attracting slow manifold and $\mathcal{S}^r_{\varepsilon}$ is a canonical repelling slow manifold.

The expression of the intersection point $\mathbf{p}^a$ (resp. $\mathbf{p}^r$) of the canonical slow manifold $\mathcal{S}^a_{\varepsilon}$ (resp.  $\mathcal{S}^r_{\varepsilon}$)  with the switching plane $\{x=0\}$, follows straightforwardly.
\end{proof}

\begin{figure}[h!]
 \begin{center}
    \includegraphics[scale=0.45]{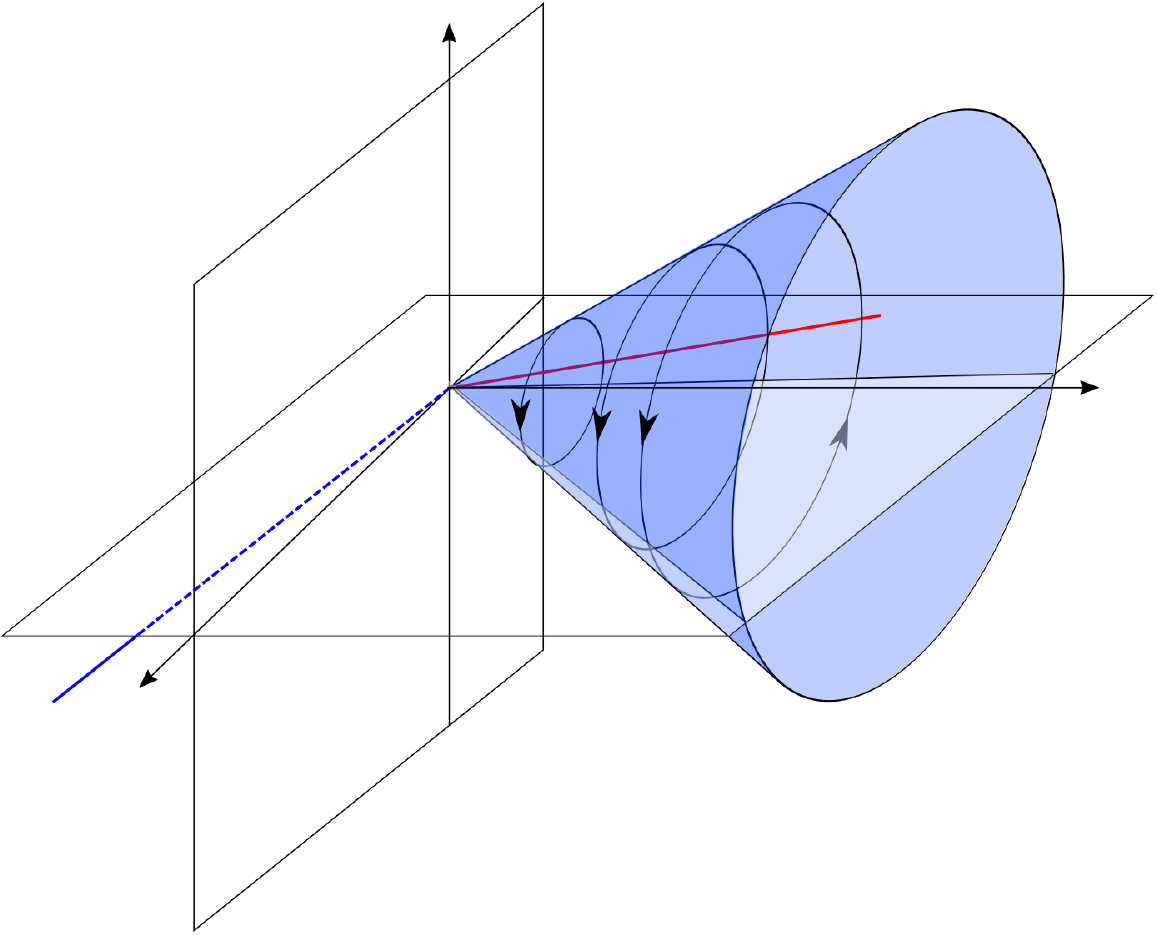}\quad 
    \includegraphics[scale=0.45]{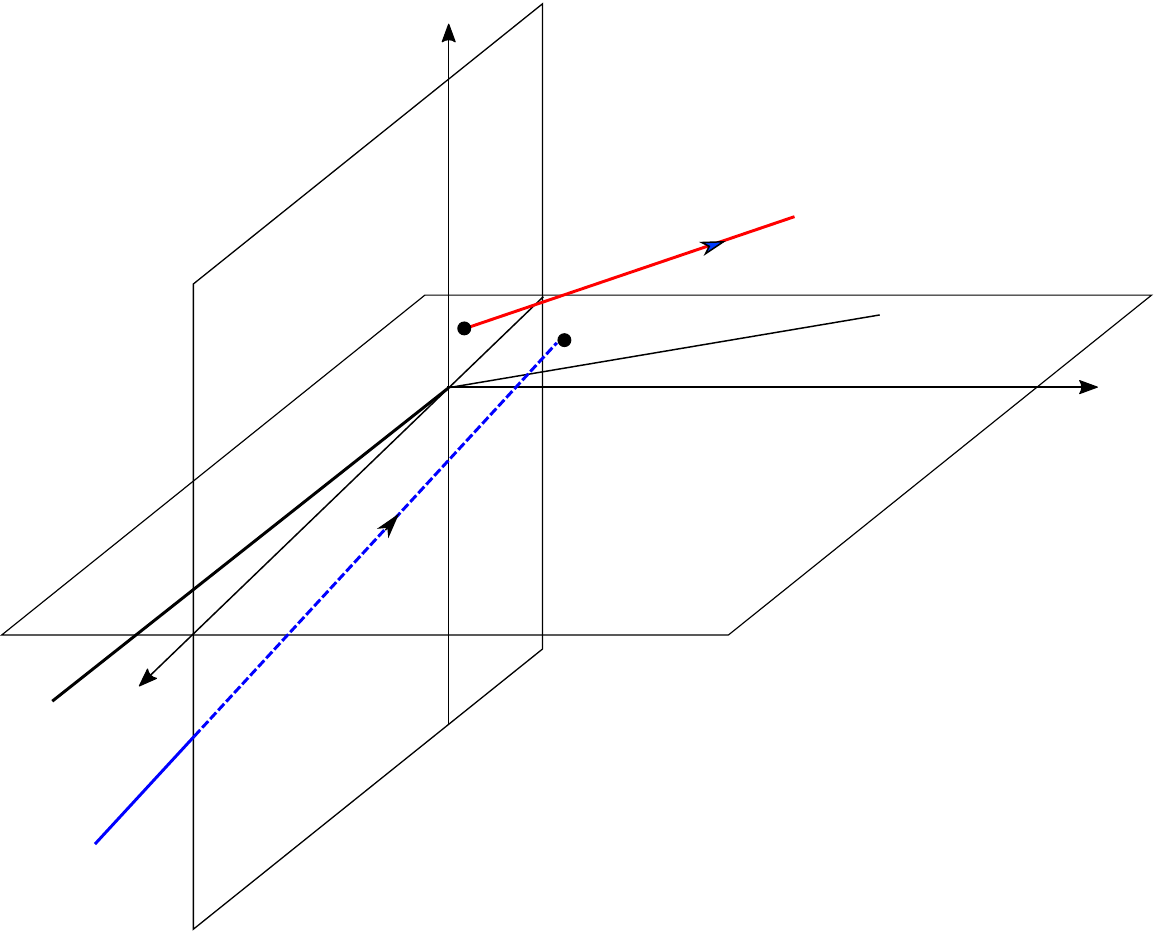}\vspace{0.2cm}\\ \;
    \begin{picture}(0,0)
      \put(-109,0){(a)}
      \put(-112,130){$x$}
      \put(-143,40){$y$}
      \put(-17,80){$z$}
      \put(-167,40){$\mathcal{S}_0^a$}
      \put(-45,100){$\mathcal{S}_0^r$}
      \put(-50,40){$\mathcal{C}$}
      \put(58,0){(b)}
      \put(51,130){$x$}
      \put(16,43){$y$}
      \put(151,80){$z$}
      \put(45,90){$\mathbf{}$}
      \put(14,22){$\mathcal{S}_{\varepsilon}^a$}
      \put(100,120){$\mathcal{S}_{\varepsilon}^r$}
      \put(61,107){$\mathbf{p}^r$}
      \put(74,81){$\mathbf{p}^a$}
      \end{picture}
 \end{center}
 \caption{\textbf{Invariant objects before and after perturbation in $\boldsymbol{\eps}$}: (a) Attracting, $\mathcal{S}_{0}^a$, and repelling, $\mathcal{S}_{0}^r$, branches of the critical manifold, and stable invariant cone $\mathcal{C}$ obtained after a supercritical Hopf-like bifurcation in the fast subsystem, for $k<m$ and $\varepsilon=0$. (b) Attracting canonical slow manifold $\mathcal{S}_{\varepsilon}^a$  and repelling canonical slow manifold $\mathcal{S}_{\varepsilon}^r$ appearing after perturbation of the critical manifold $\mathcal{S}_0$ and their intersection points, $\mathbf{p}^a=(0,-\varepsilon,m\varepsilon)$ and $\mathbf{p}^r=(0,-\varepsilon,-k\varepsilon)$, with the switching plane $\{x=0\}$. The perturbation of the cone is not represented in the figure.} \label{fig:invsets}
\end{figure}
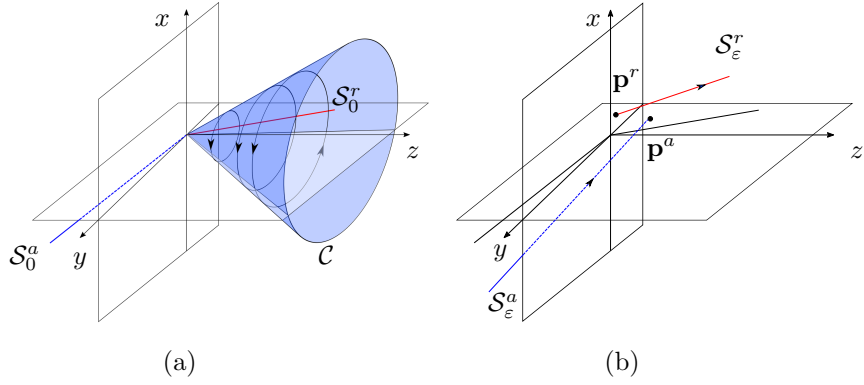

Note that both canonical slow manifolds, $\mathcal{S}_{\eps}^a$ and $\mathcal{S}_{\eps}^r$ (referred as slow manifolds if no confusion arises), do not connect on the plane $\{x=0\}$, since from~\eqref{eq:S.eps} the respective intersection points, $\mathbf{p}^a$ and $\mathbf{p}^r$, remain at a distance $(m+k)\eps$ (see~\cref{lema:inv.sets.2r}). 

In~\cref{thm:du} we will see that this phenomenon forces the delayed loss of stability to behave in a trivial way.
Indeed, consider the local expression of the flow $\mathbf{u}(t;0,(x_0,y_0,z_0))$
with respect to the initial conditions $(x_0,y_0,z_0)$ given by equations \eqref{eq:xL} and \eqref{eq:xR} in~\cref{app:solutions}, depending on $x_0<0$ and $x_0>0$, respectively. For any $t>0$, let $\pi(z_0,t)$ be the plane $\{z=z_0+\varepsilon t\}$. 
Given an initial condition $(0,y_0,z_0)$ close to $\mathbf{p}^r$, we define the {\it distance along the plane} $\pi(z_0,t)$ from the solution $\mathbf{u}(t;0,(0,y_0,z_0))$ to the slow manifold $\mathcal{S}_{\varepsilon}^r$ as
\begin{equation}\label{def:distance1}
d_R(t,(0,y_0,z_0))=\| \mathbf{u}(t;0,(0,y_0,z_0)) - (\hat{x}(t),\hat{y}(t),\hat{z}(t)) \|,
\end{equation}
where $t>0$ and $(\hat{x}(t),\hat{y}(t),\hat{z}(t))$ is the intersection between $\mathcal{S}_{\varepsilon}^r$ and $\pi(z_0,t)$. Similarly, we can define the {\it distance along the plane} $\pi(z_0,t)$ from the solution $\mathbf{u}(t;0,(0,y_0,z_0))$ to the slow manifold $\mathcal{S}_{\varepsilon}^a$ as
\begin{equation}\label{def:distance2}
d_L(t,(0,y_0,z_0))=\| \mathbf{u}(t;0,(0,y_0,z_0)) - (\hat{x}(t),\hat{y}(t),\hat{z}(t)) \|,
\end{equation}
where $t<0$ and $(\hat{x}(t),\hat{y}(t),\hat{z}(t))$ is the intersection between $\mathcal{S}_{\varepsilon}^a$ and $\pi(z_0,t)$.

\begin{lemma}\label{lema:dist}
Let $\mathbf{u}(t;0,(0,y_0,z_0))$ be the solution of system \eqref{sys:2pieces}-\eqref{eq:x-nullcina1} with initial condition $(0,y_0,z_0)$. For $t>0$, as $\mathbf{u}(t;0,(0,y_0,z_0))$ remains in the half-space $\{x>0\}$, the distance along the plane $\pi(z_0,t)$ from $\mathbf{u}(t;0,(0,y_0,z_0))$ to  $\mathcal{S}_{\varepsilon}^r$ satisfies 
\begin{equation} \label{eq:dist_2d}
d_R(t,(0,y_0,z_0)) =  d_R(0,(0,y_0,z_0)) \mathrm{e}^{\frac{k}{2}t}\frac{\sqrt{C(t)}}{\sqrt{C(0)}}, 
\end{equation}
where $\frac{1}{16} <|C(t)|<2$. 

For $t<0$, as $\mathbf{u}(t;0,(0,y_0,z_0))$ remains in the half-space $\{x<0\}$, the distance along the plane $\pi(z_0,t)$ from $\mathbf{u}(t;0,(0,y_0,z_0))$ to  $\mathcal{S}_{\varepsilon}^a$ satisfies 
\begin{equation} \label{eq:dist_2dD}
d_L(t,(0,y_0,z_0)) =  d_L(0,(0,y_0,z_0)) \mathrm{e}^{-\frac{m}{2}t}\frac{\sqrt{D(t)}}{\sqrt{D(0)}}, 
\end{equation}
where $\frac{1}{16} <|D(t)|<2$. 
\end{lemma}
\begin{proof}
From expression \eqref{eq:xR}, we write the expression of $\mathbf{u}(t;0,(0,y_0,z_0))$ on the half-space $\{x>0\}$ as
\begin{equation}\label{eq:x0}
\begin{pmatrix} x \\ y \\ z \end{pmatrix} =
\begin{pmatrix} k\varepsilon \\ k^2\varepsilon-\epsilon \\ 0 \end{pmatrix} +
 z \begin{pmatrix} 1 \\ k \\ 1 \end{pmatrix} + \vartheta \mathrm{e}^{\frac k 2 t}
 \begin{pmatrix} \cos(\theta_1+\xi_k t) \\ \cos(\theta_2+\xi_k t) \\ 0 \end{pmatrix},
\end{equation}
where 
\begin{align*}
\vartheta^2 &= \frac{4}{4-k^2}((y_0+\epsilon)^2+(z_0+k\epsilon)(z_0-ky_0)),\\
\tan\theta_1 &= \frac{k(z_0+k\epsilon)-2(y_0+\epsilon)}{\sqrt{4-k^2}(z_0+k\epsilon)},\\
\tan\theta_2 &= \frac{(z_0+k\epsilon)(k^2-2)-k(y_0+\epsilon)}{\sqrt{4-k^2}(k(z_0+k\epsilon)-(y_0+\epsilon))},
\end{align*}
and $\xi_k$ is given in~\cref{app:solutions}.

From \eqref{eq:x0}, the evolution of an orbit in the region $\{x>0\}$ can be obtained as the evolution of a point $(\hat{x},\hat{y},z)$, with $\hat{x}= k\epsilon+z$ and $\hat{y}=k^2\epsilon-\epsilon+k z$, over the repelling branch of the slow manifold $S_{\varepsilon}^r$, plus an oscillatory term with increasing amplitude. Hence, the evolution of the distance from the orbit to the slow manifold is computed as
\[
\sqrt{ (x-\hat{x})^2 + (y-\hat{y})^2} = \vartheta \mathrm{e}^{\frac{k}{2}t}\sqrt{\cos^2(\theta_1+\xi_k t)+\cos^2(\theta_2+\xi_k t)}.
\]
Taking $C(t)=\cos^2(\theta_1+\xi_k t)+\cos^2(\theta_2+\xi_k t)$ and considering different trigonometric identities, we can rewrite $C(t)$ as
\[
C(t) = 1+\frac{1}{2}\sqrt{\frac{k^2+3}{2}}\sin\left(\sqrt{4-k^2}t + \frac{\pi}{2} + \varphi \right),
\]
where
\[
\tan\varphi = \tan\frac{\theta_1+\theta_2}{2}.
\]
Hence, given that the sine function lies between $-1$ and $1$ and that $0<k<2$, we obtain $\frac{1}{16}<|C(t)|<2$. Finally, we take the initial distance $d_R(0,(0,y_0,z_0))=\vartheta\sqrt{C(0)}$, which proves \eqref{eq:dist_2d}. The expression of the distance from the solution $\mathbf{u}(t;0,(0,y_0,z_0))$ along the plane $\pi(z_0,t)$ to the attracting slow manifold $\mathcal{S}_{\varepsilon}^a$ given in \eqref{eq:dist_2dD} follows in a similar way. This concludes the proof of the lemma.
\end{proof}

From \eqref{eq:dist_2d}-\eqref{eq:dist_2dD}, both distances $d_L(t,(0,y_0,z_0))$ and $d_R(t,(0,y_0,z_0))$ depend both on time $t$ and on the initial condition $\mathbf{p}=(0,y_0,z_0)$. As the initial condition gets closer to $\mathbf{p}^r$ along a radial segment (see~\cref{fig:tubular}), the initial distance becomes smaller, and the orbit through $\mathbf{p}$ takes more time to exit from a tubular neighbourhood of  $\mathcal{S}_{\varepsilon}^r$. 
The magnitude $\delta$ of the radius of this tubular neighbourhood is not relevant as long as it remains invariant throughout the analysis, so without loss of generality we consider $\delta=1$. Let $t^*$ be the exit time of the orbit starting at $\mathbf{p}$. Then, the orbit leaves the neighbourhood of $\mathcal{S}_{\varepsilon}^r$ at point $\mathbf{p}_o$ with $z$-coordinate being $z_o=z_0+\varepsilon t^*$. In the particular case where $\mathbf{p}=\mathbf{p}^a$, we write $z_d$ for $z_o$.

\begin{figure}[ht]
\begin{center}
\includegraphics[scale=0.75]{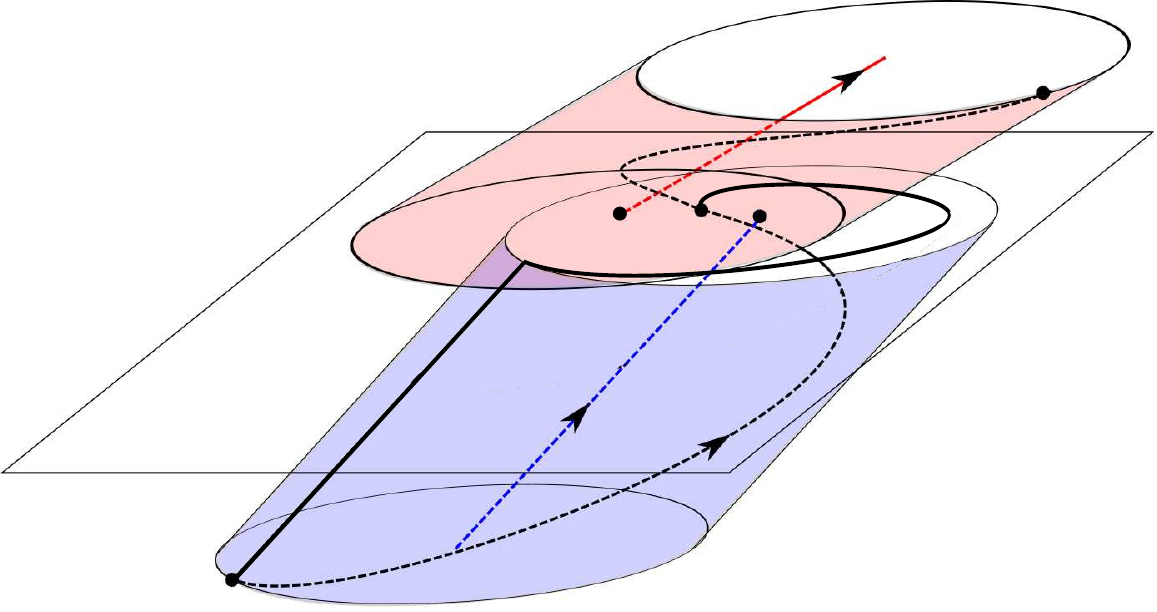}
\begin{picture}(0,0)
\put(0,100){$\{x=0\}$}
\put(-86,84){{\small $\mathbf{p}^a$}}
\put(-133,84){{\small $\mathbf{p}^r$}}
\put(-112,79){$\mathbf{p}$}
\put(-37,118){$\mathbf{p}_o$}
\put(-220,5){$\mathbf{p}_i$}
\put(-190,40){$\mathcal{R}$}
\put(-123,40){$\mathcal{S}_{\varepsilon}^a$}
\put(-87,118){$\mathcal{S}_{\varepsilon}^r$}
\end{picture}
\end{center}
\caption{\textbf{Tubular neighbourhoods of the slow manifolds.} Magnification of~\cref{fig:invsets}(b) and representation of the tubular neighbourhood of $\mathcal{S}_{\varepsilon}^a$ (in blue) and that of $\mathcal{S}_{\varepsilon}^r$ (in red). The orbit through the initial condition $\mathbf{p}=(0,y_0,z_0)$ it is also represented in backward and forward time, together with the entry point $\mathbf{p}_i$ and the exit one $\mathbf{p}_o$ (dashed black curve through $\mathbf{p}$). Solid black line named $\mathcal{R}$ represents a segment parallel to $S_{\varepsilon}^a$ lying on the tubular neighbourhood. Solid spiral lying on the plane $\{x=0\}$ represents the image by the flow of the segment $\mathcal{R}$.
}\label{fig:tubular}
\end{figure}
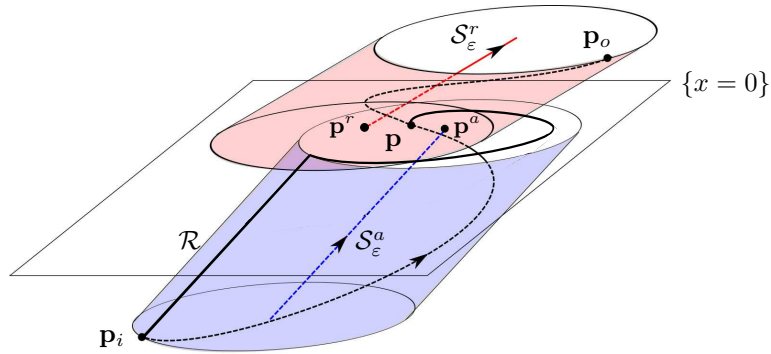

On the other hand, an inverse process takes place in the half-space $\{x<0\}$, where orbits tend to the attracting branch of the slow manifold, reaching the tubular neighbourhood of radius $\delta=1$ around $\mathcal{S}_{\varepsilon}^a$ at some point denoted by $\mathbf{p}_i$ with $z$-coordinate $z_i$. From this point, the orbit continues tending to $\mathcal{S}_{\varepsilon}^a$ and reaching the plane $\{x=0\}$ at $\mathbf{p}$ (see~\cref{fig:tubular}). If $z_i$ is negative enough, then the orbit approaches the switching plane $\{x=0\}$ very close to $\mathbf{p}^a$.

By construction, if we consider a segment $\mathcal{R}$ along the tubular neighbourhood of $\mathcal{S}_{\varepsilon}^a$, and we continue the flow generated by each point in $\mathcal{R}$, then we obtain a smooth manifold that intersects the switching plane $\{x=0\}$ at a curve spiralling towards $\mathbf{p}^a$ (see~\cref{fig:tubular}). In~\cref{fig:in-out_manifold_segment}(a), we compute this curve by taking different initial conditions on $\mathcal{R}$, which is at distance $\delta=1$ to $\mathcal{S}_{\varepsilon}^a$. We note that, the closer a point is on the spiral to $\mathbf{p}^a$, the smaller $z_i$ is and, hence, the closer $z_o$ is to $z_d$. In~\cref{fig:in-out_manifold_segment}(b), we represent the relation between $|z_i|$ and $z_o$ obtained by considering different initial conditions on $\mathcal{R}$. We call this relation the \textit{way-in/way-out function} since its behavior and properties are compatible with classical way-in/way-out functions considered in smooth slow-fast systems, see e.g.~\cite{krupa2010}. Note that the way-in/way-out function asymptotically tends to the constant value $z_d$, which can be considered as the maximal delay. Since $z_d>0$, this implies that the PWL system that we are considering exhibits a delayed loss of stability. Next, we describe how the maximal delay $z_d$ behaves as $\varepsilon$ tends to zero.

\begin{figure}[h!]
\centering
\begin{tabular}{cc}
    \includegraphics[scale=0.25]{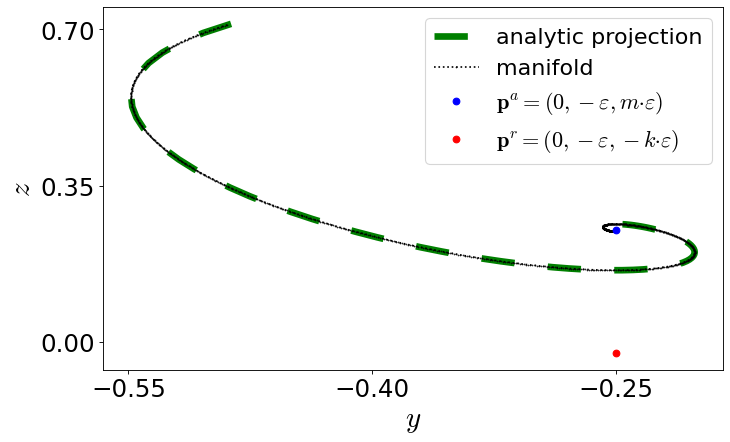}
    &
    \includegraphics[scale=0.25]{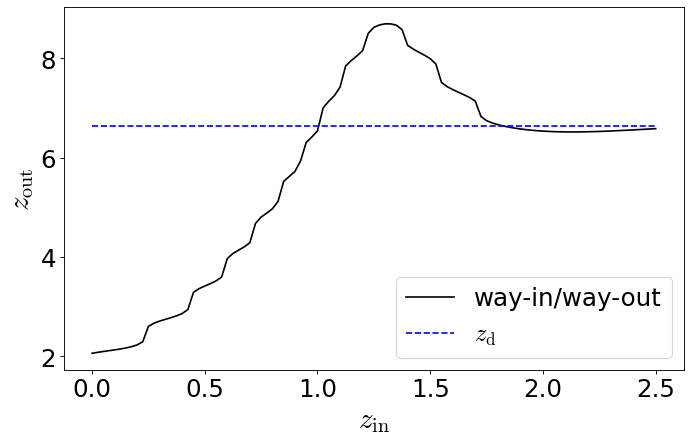} \\
    (a) & (b)
\end{tabular}
\caption{\textbf{Graph of the way-in/way-out function}. Panel (a) displays the spiral formed by the intersection points with the plane $\{x=0\}$ of the orbits starting in the segment $\mathcal{R}=\{(z-m\eps,\eps-m(z-m\eps)+1,z):\ z<0\}$ and represented in~\cref{fig:tubular}. The endpoint of the spiral is given by the point $\mathbf{p}^a$, plotted in blue. The red dot corresponds to the point $\mathbf{p}^r$. In panel (b), we solve the way-in/way-out function for points on the spiral, with initial conditions on the segment $\mathcal{R}$. The blue dotted line corresponds to the maximal delay $z_d$. The values of the remaining parameters are $k=0.1$,
$m=1$, and $\epsilon=0.25$.}
\label{fig:in-out_manifold_segment}
\end{figure}

\begin{theorem}\label{thm:du}
Consider system \eqref{sys:2pieces}-\eqref{eq:x-nullcina1}. Then, the maximal delay $z_d$ satisfies 
\begin{equation} \label{eq:zd2d}
-\frac{2\epsilon}{k}\ln\left(\frac{2\sqrt{2} (m+k)}{\sqrt{4-k^2}}\epsilon\right) + m\epsilon 
< z_d < 
-\frac{2\epsilon}{k}\ln\left(\frac{(m+k)}{2\sqrt{4-k^2}} \epsilon\right) + m\epsilon.
\end{equation}
In particular, $z_d$ tends to zero as $\varepsilon$ tends to zero. 
\end{theorem}

\begin{proof}
Consider $\gamma$ to be the orbit through $\mathbf{p}^a=(0,-\epsilon,m\epsilon)$. To compute the escape point of $\gamma$ from the $\mathcal{S}_{\varepsilon}^r$ tubular neighbourhood, we equate equation \eqref{eq:dist_2d} to $1$, which is the radius of the tubular neighbourhood. Since  $\frac{1}{16}<C(t)<2$, we obtain that $\frac{1}{16}\vartheta^2\mathrm{e}^{kt_d}<1< 2\vartheta^2\mathrm{e}^{kt_d}$. Therefore, the time of flight from $\mathbf{p}^a$ to the escape point satisfies
$$
-\frac{2}{k}\ln\left(\sqrt{2}\vartheta\right) < t_d < -\frac{2}{k}\ln\left(\frac{1}{4}\vartheta\right).
$$
Since the maximal delay is given by $z_d=\epsilon t_d+m\epsilon$, expression \eqref{eq:zd2d} holds.
Moreover, by the squeeze theorem, $z_d$ tends to zero as $\varepsilon$ tends to zero. 
\end{proof}

Therefore, \cref{lema:inv.sets.2r}  implies that slow-fast PWL systems with two zones do not allow to establish a connection between canonical attracting and repelling slow manifolds, that is, between $\mathcal{S}_{\eps}^a$ and $\mathcal{S}_{\eps}^r$. Hence, these systems present a delayed loss of the stability with the maximal delay tending to zero with $\eps$, as we have proved in \cref{thm:du}. We next modify function $f$ in system~\eqref{sys:2pieces} in order to allow for the connection, which presents a behavior closer to what happens in smooth slow-fast systems.

\section{Three-regions system} \label{sec.3regimes}

In this section, we introduce a minimal PWL system also presenting a slow-passage through a Hopf-like bifurcation, but for which the maximal delay does not tend to zero with $\eps$. For this purpose, we add a small linearity zone to replace the point of non-normal hyperbolicity at the origin. Hence, we add an extra segment in the critical manifold~\eqref{eq:x-nullcina1}. This strategy of considering a 3-piece critical manifold instead of a 2-piece one has proven useful when studying fine slow-fast phenomena related to non-normal hyperbolicity, in particular to approximate a smooth slow-fast system near a quadratic fold of the critical manifold; see~\cite{desroches2016, fernandez2016canard} for more details on this approach in the context of canard-explosive systems. Given that delayed Hopf bifurcation is akin to a canard phenomenon (with different characteristics though), this further motivates us to add this extra segment to the critical manifold of system~\eqref{sys:2pieces}.
This segment defines a new central linear region with boundaries $\rho < 0 < \mu$, chosen such that the orbit following the attracting branch $\mathcal{S}_\eps^a$ at the switching plane $\{x=\rho\}$ approaches the switching plane $\{x=\mu\}$ at the same point that $\mathcal{S}_\eps^r$ does. Thus, the connection between both slow manifolds happens, hence the slow-passage behavior.

Let us consider system \eqref{sys:2pieces} where the function $f(x)$ is now given by
\begin{equation}\label{eq:f.3pw}
f(x) = \left\{
\begin{array}{cc}
-mx & x< \rho, \\
lx+n & \rho \leq x < \mu, \\
kx & x\geq\mu,
\end{array}
\right.
\end{equation}
with parameters $l$ and $n$ fixed as 
\begin{equation} \label{eq:l.n}
l =\frac{m\rho+k\mu}{\mu-\rho}, \quad
n = -\rho\mu\frac{k+m}{\mu-\rho},
\end{equation}
to ensure the continuity of the vector field. The slope $l$ is the weighted mean of the slopes $-m$ and $k$ with weights $\frac{-\rho}{\mu-\rho}$ and $\frac{\mu}{\mu-\rho}$, respectively, and then $-2<-m<l<k<2$. Notice that, since parameters $\rho <0 <\mu$ organize the boundaries between the different linear regions, in the particular case where these regions are centered, that is when $\mu=-\rho$, then $l$ is the arithmetic mean of the slopes $-m$ and $k$. 

As solutions cross the plane $\{x=\rho\}$, their local expression changes from the left-side expression~\eqref{eq:xL} to the central one \eqref{eq:xC}. Similarly, as solutions cross the plane $\{x=\mu\}$, their local expression changes from the central to the right-side expression \eqref{eq:xR}. 

As in the two regions case, slow manifolds perturb from the critical manifold when $\eps>0$ is small enough, see~\cref{lema:inv.sets.2r}. Then, the segments $\{(x,-mx-\varepsilon, x+m\varepsilon),\, x< \rho\}$ and $\{(x,kx-\varepsilon, x-k\varepsilon),\, x> \mu\}$ are part of the attracting and of the repelling  canonical slow manifolds, $\mathcal{S}_{\varepsilon}^a$ and $\mathcal{S}_{\varepsilon}^r$, respectively. These segments intersect with the planes $\{x=\rho\}$ and $\{x=\mu\}$ at points 
$\tilde{\mathbf{p}}^a=(\rho,-m\rho-\varepsilon,\rho+m\varepsilon)$ and 
$\tilde{\mathbf{p}}^r=(\mu,k\mu-\varepsilon,\mu-k\varepsilon)$, respectively.
On the other hand, the segment $\{(x,xl+n-\varepsilon, x-\varepsilon l),\, \rho < x < \mu\}$ is invariant under the flow of the system defined in the central region, and it is either part of the attracting ($l<0$) or the repelling ($l>0$) slow manifold, respectively.

To be able to connect both lateral slow manifolds through the central region, we impose that the solution in the central region~\eqref{eq:xC} passing through $\tilde{\mathbf{p}}^a$ reaches the boundary $\{x=\mu\}$ at $\tilde{\mathbf{p}}^r$. This yields
\begin{equation}\label{eq:connect_cond1}
\mathbf{u}_{C}(\hat{t};0,\tilde{\mathbf{p}}^a) = \tilde{\mathbf{p}}^r.
\end{equation}
This boundary condition provides three different equations, corresponding to the three coordinates 
of the solution in the central region \eqref{eq:xC}, with six unknowns, $(\hat{t},m, k,\rho,\mu,\varepsilon)$. Next, we solve this system of equations by obtaining the first three unknowns in terms of the others.

\begin{lemma} \label{lema:m.k}
Consider system \eqref{sys:2pieces} with $f$ given by \eqref{eq:f.3pw} and $\rho<0<\mu$. Assuming
\begin{equation} \label{eq:m}
m = \frac{\mu - \rho}{\epsilon}- k - \sqrt{\ln^2\left\lvert\frac{\rho}{\mu}\right\rvert+\pi^2}
\end{equation}
and 
\begin{equation} \label{eq:k}
k = \frac{-1}{\mu-\rho}\left(\frac{2\ln\left\lvert\frac{\rho}{\mu}\right\rvert(\rho-\mu)}{\sqrt{\ln^2\left\lvert\frac{\rho}{\mu}\right\rvert+\pi^2}} + \frac{\rho}{\epsilon}\big(\mu-\rho-\epsilon\sqrt{\ln^2\left\lvert\frac{\rho}{\mu}\right\rvert+\pi^2}\big) \right),
\end{equation}
then equation~\eqref{eq:connect_cond1} is satisfied with the time of flight of the solution in the central region being equal to 
\begin{equation} \label{eq:t}
\hat{t}= \frac{\mu-\rho}{\epsilon}-(k+m).
\end{equation}
\end{lemma}

\begin{proof}
The time of flight~\eqref{eq:t} from $x=\rho$ to $x=\mu$ can be obtained from the third equation in \eqref{eq:connect_cond1}
(corresponding to the $z$ component), where $z(\hat{t})=m\varepsilon+\rho+\varepsilon\hat{t}=-k\varepsilon+\mu$. Moreover, the first and second equations in \eqref{eq:connect_cond1} yield a linear system in terms of $\mathrm{e}^{\frac{l}{2}\hat{t}}\cos(\xi_l\hat{t})$ and $\mathrm{e}^{\frac{l}{2}\hat{t}}\sin(\xi_l\hat{t})$, whose solution is given by 
$$
\begin{pmatrix}
\mathrm{e}^{\frac{l}{2}\hat{t}}\cos(\xi_l\hat{t}) \\
\mathrm{e}^{\frac{l}{2}\hat{t}}\sin(\xi_l\hat{t})
\end{pmatrix}
=
\frac{\rho}{\mu}
\begin{pmatrix}
1 \\ 0
\end{pmatrix}.
$$
These expressions are equivalent to
\begin{equation} \label{eq:l.t}
l = \frac{2}{\hat{t}}\ln\left\lvert\frac{\rho}{\mu} \right\rvert \quad \text{and} \quad
\frac{\sqrt{4-l^2}}{2}\hat{t} = \pi + 2\pi q,  \text{ with}\quad q\in\mathbb{Z}.
\end{equation}

Considering $q=0$ and merging the first equation in \eqref{eq:l.t} with the second one, we obtain that the latter can be written as 
$$
\hat{t}^2-\ln^2\left\lvert\frac{\rho}{\mu} \right\rvert= \pi^2.
$$
Then, replacing $\hat{t}$ in the last expression by its value given by \eqref{eq:t}, and isolating $m$ in the resulting expression, we obtain equation \eqref{eq:m}. 

Similarly, equation \eqref{eq:k} is obtained by equating the first expressions in \eqref{eq:l.n} and in \eqref{eq:l.t}, replacing therein the value of $m$ obtained in equation \eqref{eq:m}, and then isolating $k$. 

To ensure that the local solution $\mathbf{u}_C(t;0,\tilde{\mathbf{p}}^a)$ is contained in the central region, that is $\rho\leq x_C(t;0,\tilde{\mathbf{p}}^a)\leq \mu$, for all $t\in[0,\hat{t}]$, we study the sign of the derivative 
$$
x_C'(t;0,\tilde{\mathbf{p}}^a) = \varepsilon\Big( 1+\frac{m+l}{\xi_l}\sin(\xi_l t)e^{\frac{l}{2}t}\Big).
$$
Then, since $m+l>0$, this function is positive for all $\sin(\xi_l t)\geq 0$, that is for all $t\in[0,\frac{\pi}{\xi_l}]=[0,\hat{t}]$. Thus, $x_C(t;0,\tilde{\mathbf{p}}^a)$ is a monotone function in this interval with $x_C(0;0,{\tilde{\mathbf{p}}}^a)=\rho$ and $x_C(\hat{t};0,{\tilde{\mathbf{p}}}^a)=\mu$.
\end{proof}

\begin{remark}
In some contexts, it is usual to set the slope in each region rather than their boundaries, as it is done in~\cref{lema:m.k}.
In order to set the slopes, we need to study the Jacobian matrix of the function $\left(m(\rho,\mu),k(\rho,\mu)\right)$. 
Then, by the Inverse Function Theorem, we can invert the previous functions if the expression
 \[
 \frac{2\pi^2}{\frac{-\rho}{\mu}A}
 + \frac{1}{\epsilon(\frac{-\rho}{\mu}+1)^2}\Big(
 (\mu-\rho)\sqrt{A} - \epsilon A \Big),
 \]
with $A=ln^2\left\lvert\frac{\rho}{\mu}\right\rvert+\pi^2$,
is different from $0$. However, all its terms are positive except the last one. Hence, it may happen that for some values we cannot guarantee the existence of the inverse function. In the particular case where parameters satisfy 
 \[
 (\mu-\rho)\sqrt{A} > \epsilon A,
 \]
we can guarantee the existence of the inverse in a neighbourhood of $(\rho,\mu)$. In fact, if the system is centered, $\mu=-\rho$, the previous inequality becomes $\frac{2\mu}{\pi}>\epsilon$, which amounts to ensure that $m$ and $k$ are positives, see \eqref{eq:m}-\eqref{eq:k}.
\end{remark}

\cref{lema:m.k} provides conditions on the parameters $m$ and $k$ to ensure the connection between the attracting and repelling slow manifolds. These conditions are sufficient, but not necessary. Indeed, for some parameter choices, we can obtain some loops in the central region by considering other values for $q$ in \eqref{eq:l.t} (see~\cref{fig:other.connections}(a)), or cross the boundary $x=\mu$ at least one time before the connection is made (see~\cref{fig:other.connections}(b)). 
These situations provide different ways to connect the slow manifolds. Nevertheless, the full study of how to make the connection goes beyond the scope of the present work, given that our aim is to find a minimal system exhibiting any simple connection between the two slow manifolds. 

\begin{figure}[h]
    \centering
    \begin{tabular}{cc}
    \includegraphics[scale=1.]{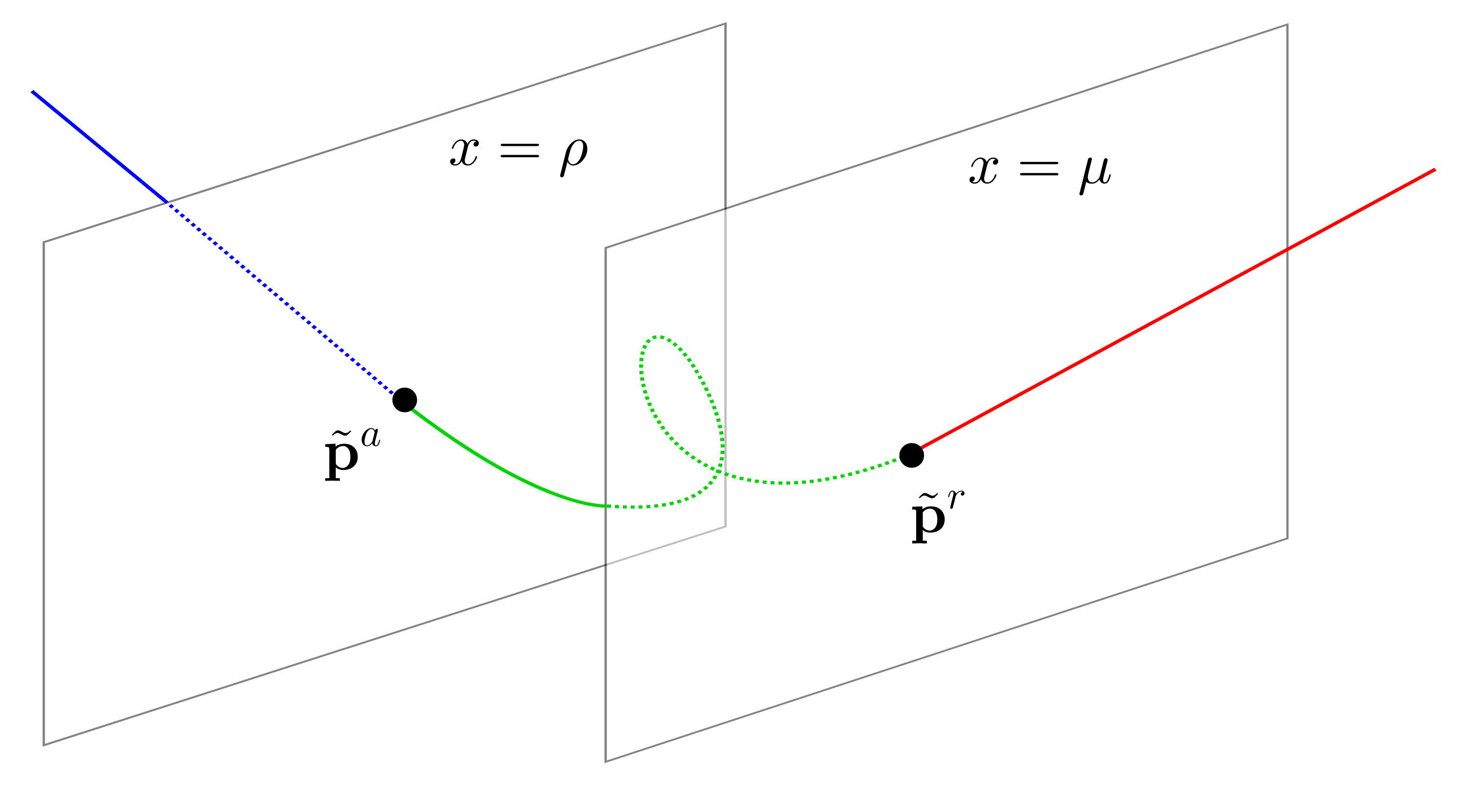} & \includegraphics[scale=1.]{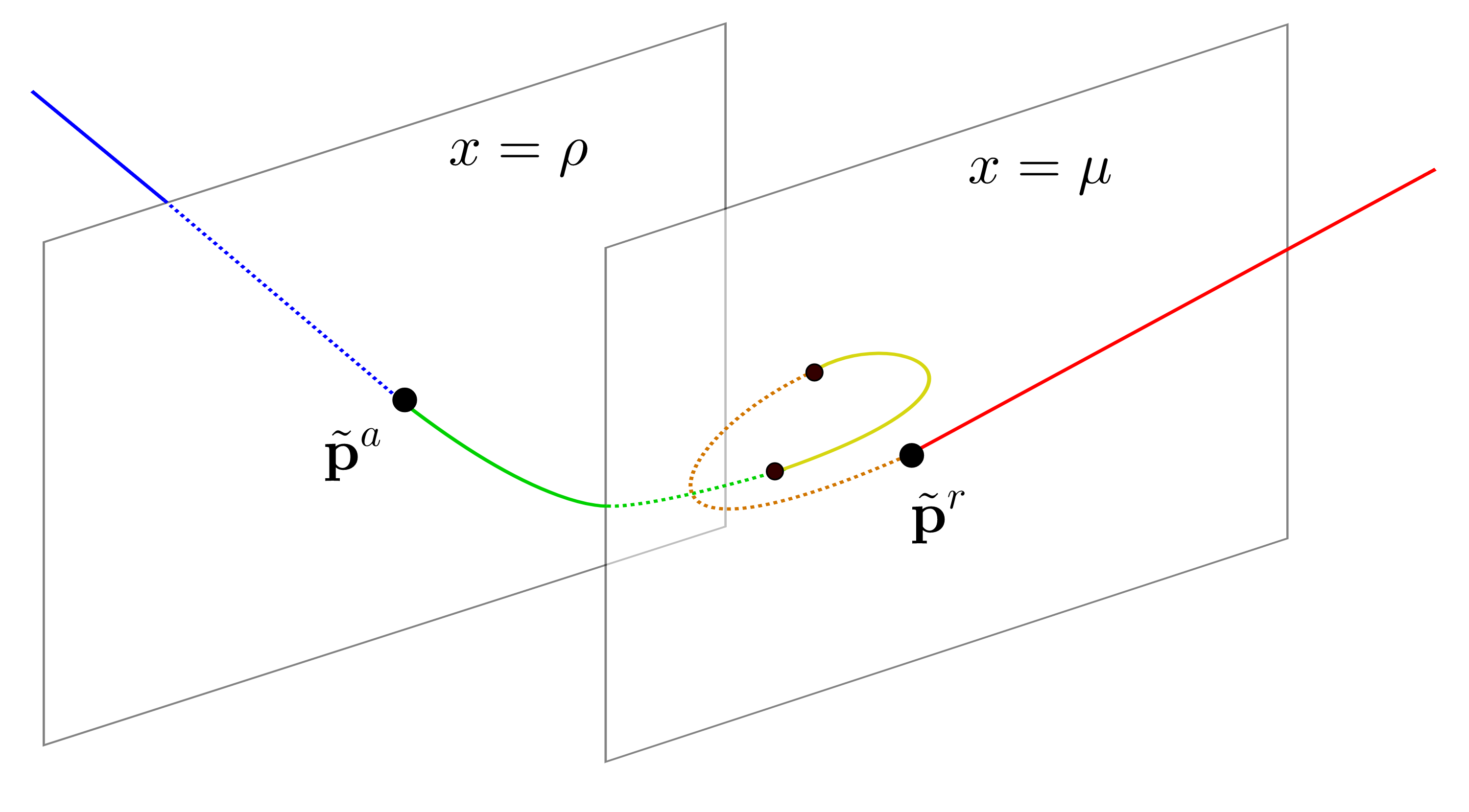} \\
    (a) & (b) 
    \end{tabular}
    \caption{\textbf{Connecting the slow manifolds.} Other possible connections that can appear between the attracting (blue line) and the repelling (red line) slow manifolds.
    The connection can be given; (a) with a loop in the central region; (b) crossing $x=\mu$ at least one time before connecting with the repelling slow manifold.}
    \label{fig:other.connections}
\end{figure}

\begin{theorem}
\label{thm:3r.simple}
Consider system \eqref{sys:2pieces} with $f$ given by \eqref{eq:f.3pw} and let $\rho<0<\mu$. Let $m$ and $k$ satisfy \eqref{eq:m} and \eqref{eq:k}, respectively. Then, the maximal delay tends to infinity as $\varepsilon$ tends to zero. Moreover, the way-in/way-out function is asymptotic to the line $z_{\mathrm{out}} = \frac{m}{k}z_{\mathrm{in}} - \frac{2\eps}{k}\ln\left\lvert\frac{\rho}{\mu}\right\rvert$.
\end{theorem}

\begin{proof}
The assertion about the behavior of the maximal delay as $\varepsilon$ tends to zero follows immediately from~\cref{lema:m.k}, since under these conditions both slow manifolds $\mathcal{S}_{\varepsilon}^{a,r}$ connect, which implies that the orbit along the attracting slow manifold continues along the repelling slow manifold and the delay can be considered infinite.

On the other hand, the slope of the way-in/way-out function can be approximated as follows. Consider a tubular neighbourhood of radius $\delta$ around the attracting slow manifold, and a similar neighbourhood around the repelling one; without loss of generality, we can take $\delta=1$ on both sides. Let $\mathbf{p}_i=(x_i,y_i,z_i)$ be an initial condition (also called entry point) located on the left tubular neighbourhood around the attracting slow manifold, that is, at a distance $\delta$ to $\mathcal{S}_{\eps}^a$, see \eqref{eq:dist_2dD}. Let $\mathbf{p}_{\rho}=(x_{\rho},y_{\rho},z_{\rho})$ be the point where the orbit through $\mathbf{p}_{i}$ crosses the boundary $\{x=\rho\}$ and $\mathbf{p}_{\mu}=(x_{\mu},y_{\mu},z_{\mu})$ the one crossing the  $\{x=\mu\}$ boundary. Finally, consider $\mathbf{p}_o=(x_o,y_o,z_o)$ be the point of the orbit located on the right tubular neighbourhood, that is, around the repelling slow manifold (also called exit point). The slope of the way-in/way-out function is given by the ratio of the $z$-component of the entry and exit points, respectively, that is, $(z_o-z_{\mu})/(z_{\rho}-z_i)$. By definition, $z_i = \varepsilon t_i+z_{\rho}$ and $z_o = \varepsilon t_o+z_{\mu}$, where $-t_i>0$ is the time necessary to reach $\mathbf{p}_{\rho}$ from the initial condition $\mathbf{p}_{i}$ and $t_o$ is the time required to reach $\mathbf{p}_o$ starting at $\mathbf{p}_{\mu}$. Hence,
\begin{equation} \label{eq.z1_z2}
    z_o-z_{\mu} = \frac{t_o}{-t_i}(z_{\rho}-z_i).
\end{equation}
On the other hand, the contraction along the attracting slow manifold at time $t_i$ yields
$\delta=d_L(t_i)=d_L(0) \mathrm{e}^{-\frac{m}{2} t_i}$,
see equation \eqref{eq:dist_2dD}, while the expansion along the repelling slow manifold at time $t_o$ is given by $\delta=d_R(t_o) = d_R(0) \mathrm{e}^{\frac{k}{2}t_o}$,
see equation \eqref{eq:dist_2d}. Hence, combining both expressions, it follows $d_L(0) \mathrm{e}^{-\frac{m}{2} t_i}=d_R(0) \mathrm{e}^{\frac{k}{2}t_o}$. On the other hand, since $d_R(0)\approx d_L(0)e^{\frac{l}{2}\hat{t}}$, where $\frac{l\hat{t}}{2}$ can be obtained from \eqref{eq:l.t}, it follows $d_R(0)\approx d_L(0)\left\lvert\frac{\rho}{\mu}\right\rvert$. Therefore, we obtain that $ \frac {t_o}{-t_i}\approx\frac{m}{k} + \frac{2}{kt_i}\ln\left\lvert\frac{\rho}{\mu}\right\rvert$ and so, from $\eqref{eq.z1_z2}$ and $-\varepsilon t_i=z_{\rho}-z_i$ we have that
\[
z_o-z_{\mu} \approx \frac{m}{k}(z_{\rho}-z_i) - \frac{2\eps}{k}\ln\left\lvert\frac{\rho}{\mu}\right\rvert,
\]
which proves the second statement of the theorem.
\end{proof}

Given $\mu$ and $\rho$, and under assumptions of Lemma\cref{lema:m.k}, if we perturb the boundaries, then the slopes $m$, $k$ and $l$ will change. The relative position of the boundaries imposes both the sign of the slope $l$ of the critical manifold in the central region, as well as the relative size of the slopes in the other two regions, as shown in the following lemma.

\begin{lemma} \label{lema:mu_rho}
Consider the system \eqref{sys:2pieces} with $f$ given by \eqref{eq:f.3pw}. Suppose that equations \eqref{eq:m} and \eqref{eq:k} are satisfied. Then, the following statements hold:
\begin{itemize}
    \item[(a)] if $\mu<-\rho$, then $l>0$ and $k>m$,
    \item[(b)] if $\mu>-\rho$, then $l<0$ and $k<m$,
    \item[(c)] if $\mu=-\rho$, then $l=0$ and $k=m$.
\end{itemize}
\end{lemma}

\begin{proof}
Since the first equation in \eqref{eq:l.t} holds, if $\mu<-\rho$, then $l$ is positive. Consequently, from the first equation in \eqref{eq:l.n}, we have that $\mu k+\rho m>0$ and thus
$$
k > \frac{-\rho}{\mu}m > m.
$$
Following the same procedure, if we consider $\mu>-\rho$, then $l<0$, $\mu k+\rho m<0$ and
$$
k < \frac{-\rho}{\mu}m < m.
$$
Finally, in the specific case that $\mu=-\rho$, from equation \eqref{eq:l.t} we obtain that $l=0$ and thus $\mu k+\rho m=0$. Hence, $k=m$.
\end{proof}

Take the minimal PWL slow-fast system \eqref{sys:2pieces} with $f$ given by \eqref{eq:f.3pw}, and consider the fast subsystem at $\varepsilon=0$. Then, following \cite{FPT99}, the existence of a Hopf-like bifurcation (supercritical or subcritical) in the fast subsystem \eqref{eq:odeInit} can be obtained according to a magnitude ratio between the slopes of the $x$-nullcline. Our next result provides conditions to guarantee the existence, location and criticality of this bifurcation.

\begin{lemma}\label{lema:l.hopf}
Consider system \eqref{eq:odeInit} with $f$ given by \eqref{eq:f.3pw}. Then, the following assertions hold:
\begin{itemize}
    \item [(a)] If $l>0$, a Hopf-like bifurcation occurs at $z=\rho$; it is a subcritical if $l-m>0$ and supercritical otherwise.
    \item [(b)] If $l<0$, a Hopf-like bifurcation occurs at $z=\mu$; it is a subcritical if $l+k>0$ and supercritical otherwise.
\end{itemize}
\end{lemma}

\begin{proof}
Consider the 2-regions system formed by the central and left regions of \eqref{eq:odeInit}-\eqref{eq:f.3pw}. Then, the change of variables given by $x = \hat{x} + \rho$ and $y = \hat{y} - m\rho$ allows us to rewrite the system in the following form which is analyzed in  \cite{FPT99}
\begin{equation}\label{sys:aux1}
\begin{array}{l}
 \dot{\hat{x}} = \hat{f}(\hat{x})-\hat{y}, \\
 \dot{\hat{y}} = \hat{x} + a,
\end{array}
\end{equation}
with, in the present case, $a=\rho-z$, $\hat{f}(\hat{x})= -m\hat{x}$ if $\hat{x}<0$ and $\hat{f}(\hat{x})= l\hat{x}$ otherwise. Let $\text{det}_L=1$ and $t_L=-m$ be the determinant and the trace of the matrix defining the linear subsystem \eqref{sys:aux1} in $x<\rho$, and $\text{det}_C=1$ and $t_C=l$ be the determinant and the trace of the matrix in $x>\rho$. Consider also the  weighted traces
\[ 
 \gamma_{L}=\frac{t_L}{\sqrt{4\text{det}_{L}-t_L^2}}=\frac{-m}{\sqrt{4-m^2}}\quad \text{and}\quad
  \gamma_{R}=\frac{t_C}{\sqrt{4\text{det}_{C}-t_C^2}}=\frac{l}{\sqrt{4-l^2}}.
\]
Since $-2<-m<l<2$, it follows that 
\[
 D_L = \text{det}_{L} - \frac{t_L^2}{4} = \frac{4-m^2}{4} >0, \quad \text{and}\quad
 D_R = \text{det}_{C} - \frac{t_C^2}{4} = \frac{4-l^2}{4} >0. 
\] 

In the conditions of statement (a), it follows that $t_L\,t_C=-ml<0$ and $\text{sign}(\gamma_L)+\text{sign}(\gamma_R) = \text{sign}(l-m)$. According to the sign of $l-m$, we obtain different situations following Theorem 5 in \cite{FPT99}, which are explained below. 

If $l-m>0$, for $z<\rho$ we have $a>0$ and so we have an asymptotically stable equilibrium which is surrounded by a unique unstable limit cycle (case (a)). Otherwise, for $z>\rho$ we have $a<0$ and so the equilibrium becomes unstable and no limit cycles exist (case (c)). Consequently, a supercritical Hopf-like bifurcation appears at $z=\rho$.

If $l-m< 0$, for $z<\rho$ we have $a>0$ and so we have an asymptotically stable equilibrium and the system do not present limit cycles (case (b)). Otherwise, for $z>\rho$ we have $a<0$ and so the equilibrium becomes unstable and it is surrounded by a unique stable limit cycle (case (d)). Hence, a subcritical Hopf-like bifurcation appears at $z=\rho$.

Otherwise, if $l=m$, notice that in this case no criterion is established since Theorem 5 only applies when $\text{sign}(a)=\text{sign}(\gamma_C)$. In such a case, the continuity of the vector field implies that the linear systems in the regions $L$ and $C$ are identical and the minimal PWL model is equivalent to the one analyzed in~\cref{sec.2regimes}, but here with the boundary translated to $x=\mu$. This finishes the proof of statement \textit{(a)}.

If we consider the 2-regions system given by the central and right regions and proceeding in a similar way that in the proof of the statement \textit{(a)}, the proof of statement \textit{(b)} follows.
\end{proof}

Let us consider the specific case $\mu=-\rho$, where the central region is centered at the origin. Under this new assumption,~\cref{lema:mu_rho} implies that $l=0$ and $m=k$. Hence, from equations \eqref{eq:m}-\eqref{eq:k}, we obtain that
$k=m=\frac{\mu}{\epsilon}-\frac{\pi}{2}$. In particular, to ensure that  $0<k=m<2$, we need to consider that 
$\frac{\pi}2\varepsilon < \mu < \frac {\pi+4}2 \varepsilon$,  which forces the width of the central region to be of order $\varepsilon$. Notice that, for $l=0$,~\cref{lema:l.hopf} does not ensure the existence of the Hopf-like bifurcation. Nonetheless, given $\varepsilon$ and $\mu=-\rho$ lying in the specified interval, there is a slight perturbation of $\mu$ and $\rho$ such that $m$ and $k$ remain in the interval $(0,2)$. Hence, by~\cref{lema:mu_rho}, if $\mu<-\rho$, then $k>m$, and if $\mu>-\rho$, then $k<m$. Therefore, by~\cref{lema:l.hopf}, a Hopf-like bifurcation appears.  
 
This fact is illustrated in~\cref{fig:SP3d_3R}. Panel (a) shows the case $\mu=-\rho=0.15$ and $\eps=0.05$, where no Hopf-like bifurcation exists, as can be concluded from the exponential growth of the amplitude of the oscillations. However, if we slightly perturb $\rho$, taking for instance $\rho=-0.085$ (see Panel (b)), we obtain $k<m$ and a supercritical Hopf-like bifurcation appears.

\begin{figure}[h]
    \centering
    \begin{tabular}{cc}
    \includegraphics[scale=0.28]{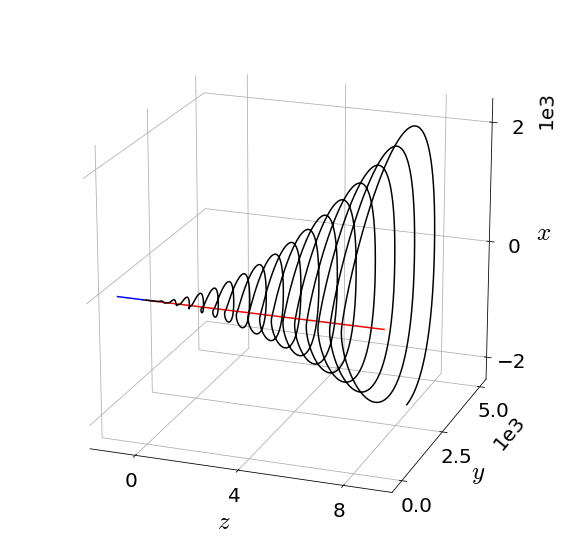}  & 
    \includegraphics[scale=0.28]{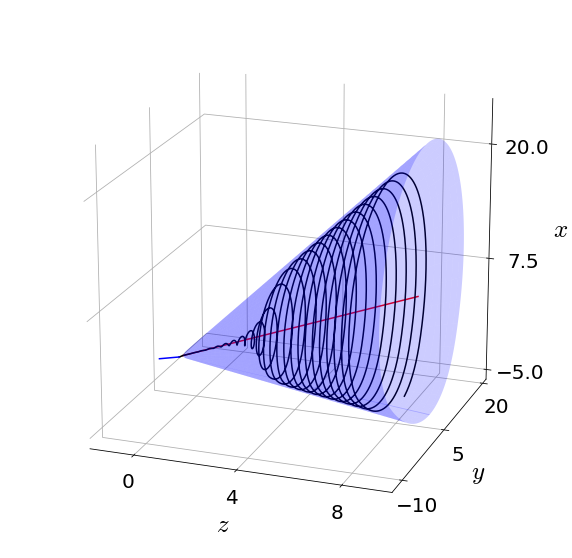} \\
    (a) & (b)
    \end{tabular}
    
    \caption{\textbf{Slow passage phenomenon through a Hopf-like bifurcation}, with initial conditions near the attracting slow manifold (blue line). The repelling slow manifold is depicted in red. The parameters are $\mu=0.15$, $\epsilon=0.05$ and (a) $\rho=-0.15$, in this case, no Hopf bifurcation occurs; (b) $\rho=-0.085$, in this case a supercritical Hopf-like bifurcation takes place defining an attracting manifold, in blue in the figure. The values of $m$ and $k$ are given by expressions \eqref{eq:m} and \eqref{eq:k}, respectively.}
    \label{fig:SP3d_3R}
\end{figure}

In~\cref{fig:in-out} we plot the way-in/way-out function for $\mu=0.15$ and $\rho=-0.085$; here, $k$ and $m$ are fixed so that the relations in~\cref{lema:m.k} are satisfied and guarantee the existence of a connection between the lateral slow manifolds. In the figure, we can see that the $z$ coordinate of the escape point from the tubular neighbourhood of the repelling slow manifold is proportional to the $z$ coordinate of the entry point to the tubular neighbourhood of the attracting slow manifold. Due to the connection between both slow manifolds, this relation is expected to persist along the whole domain of the function, as shown in~\cref{thm:3r.simple}. Nevertheless, in the simulations, this relation is not satisfied at all and the way-in/way-out function remains constant to a specific value. This unexpected behavior is due to the numerical precision in the simulations, which does not provide an exact connection. Therefore the plateau do not correspond to a buffer point (see~\cref{fig:in-out}).

\begin{figure}[h]
    \centering
    \includegraphics[scale=0.4]{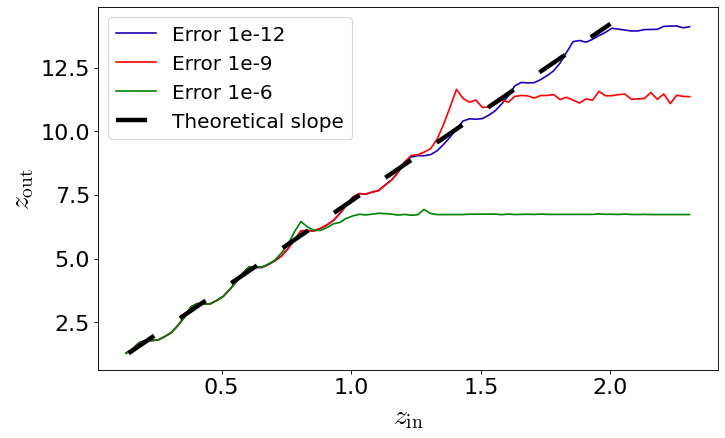}
    \caption{\textbf{Way-in/way-out function for a delayed supercritical Hopf-like bifurcation.} Plot of the way-in/way-out function in the case $\rho = -0.085$, $\mu = 0.15$, $\varepsilon=0.05$, for the parameters $k\approx 0.189$, $m\approx 1.318$ and $l\approx -0.356$ computed according to equations \eqref{eq:m}, \eqref{eq:k} and \eqref{eq:l.n}. The way-in/way-out function has been calculated for different numerical precision. The dashed curve corresponds to the asymptotic line $z_{\mathrm{out}}=\frac m k z_{\mathrm{in}}-\frac {2\varepsilon}{k}\ln{\left\lvert\frac{\rho}{\mu}\right\rvert}$ according to~\cref{thm:3r.simple}.}
    \label{fig:in-out}
\end{figure}

To support this claim, we study the relationship between computational accuracy and the asymptotic value of the graph at which it saturates.  \cref{tab:precision} relates the error in the connection, $\vartheta$, needed to obtain the observed saturation point in the way-in/way-out function, with the computational error of the numerical integrator. Using~\cref{lema:dist} and equation \eqref{eq:dist_2d}, we can bound $\vartheta$ by the extreme values of $C(t)$, that is,
$$
\sqrt{\frac{1}{2\mathrm{e}^{kt}}} < \vartheta < \sqrt{\frac{16}{\mathrm{e}^{kt}}}.
$$
Hence, using the flying time from $x=\mu$ to the escape time, we obtain the minimum and maximum values of $\vartheta$ given in~\cref{tab:precision}. As we can see, the computational precision has approximately the same order as $\vartheta$. Therefore, the saturation value that can be seen in~\cref{fig:in-out} agrees with the computational precision.

\begin{table}[h]
    \centering
    \begin{tabular}{|c|c|c|}
    \hline
    Computational Precision & min $\vartheta$ & max $\vartheta$  \\ \hline \hline 
    $10^{-12}$ & $2.78\cdot10^{-12}$ & $1.57\cdot10^{-11}$ \\ 
    $10^{-9}$ & $1.43\cdot 10^{-10}$ & $8.09\cdot 10^{-10}$ \\
    $10^{-6}$ & $1.24\cdot 10^{-6}$ & 
    $7.03\cdot10^{-6}$\\
    \hline
    \end{tabular}
    \caption{\textbf{Comparison between round-off error and error at the connection.} Comparison between the precision of the numerical integrator and the value of $\vartheta$, which provides an approximate value of the initial distance between the attracting and the repelling slow manifolds on the plane $\{x=\mu\}$. To compute the second and the third column, we use the way-in/way-out function evaluated on the attracting slow manifold at $x=\rho$, with a tubular neighbourhood of radius $1$ and the same parameters as in~\cref{fig:in-out}(a).}
    \label{tab:precision}
\end{table}

\subsection{Adding a buffer point}\label{sec:buff.point}

As described in~\cref{thm:3r.simple}, the connection between the slow manifolds yields a maximal delay that is not finite, which implies that the way-in/way-out function is unbounded. However, the presence of an equilibrium point on the repelling slow manifold can modify this behavior. In fact, such an equilibrium blocks the flow not only on the repelling slow manifold that connects to it, but also for orbits sufficiently close to this slow manifold. Hence, it forces the $z$-coordinate of the exit point of each such orbits to remain constant (see \cite{diener1991maximal,neishtadt2009stability}). This constant value is called a \textit{buffer point.} 

In order to analyze the slow passage through a Hopf-like bifurcation in the presence of a buffer point, we extend the previous PWL slow-fast system and now consider
\begin{equation} \label{eq:odeBuffer}
\begin{array}{ccl}
\dot{x} & = & f(x)-y, \\
\dot{y} & = & x-z, \\
\dot{z} & = & \epsilon(a-x),
\end{array}
\end{equation}
where $f(x)$ is still taken as in \eqref{eq:f.3pw}. 

We start by establishing conditions on the parameters in order to guarantee a connection between the attracting and the repelling slow manifolds. As the expressions of the eigenvalues become more complicated, we now assume that the matrix of the system in each region has its spectrum formed by one real eigenvalue $\lambda_j$ and a pair of complex conjugated eigenvalues $\alpha_j\pm \beta_j \, \rm{i}$, with $j\in\{L,C,R\}$; the subscripts $L,C$ and $R$ stand for the left half-space $\{x<\rho\}$, the central strip $\{\rho<x<\mu\}$ and the right half-space $\{x>\mu\}$, respectively. To simplify the notations, when $\lambda_j\alpha_j<0$ we refer to this configuration as of \textit{saddle-focus type}, when $\lambda_j\alpha_j>0$ we refer to it as of \textit{node-focus type} and when $\alpha_j=0$ we refer to it as of \textit{saddle-center type}. In particular, we consider a configuration of  saddle-center type in the central region, that is, $\alpha_C=0$. Therefore, the complex eigenvalues of the matrix of the linear system defined in the central strip
\begin{equation}\label{eq:centralsys.bfpoint}
\mathbf{\dot{u}} = 
\begin{pmatrix}
l & -1 & 0 \\
1 & 0 & -1 \\
-\epsilon & 0 & 0
\end{pmatrix}\mathbf{u}+
\begin{pmatrix}
n \\ 0 \\ \epsilon a
\end{pmatrix},
\end{equation}
are purely imaginary. We also assume the size of this central region to be order $O(\varepsilon)$. Following \cite{ProhensVich}, the attracting and repelling slow manifolds are given by the line segments each formed by the  equilibrium point and eigenvector associated with the slow eigenvalue, contained respectively in L or R, i.e.

\begin{equation}\label{def:slowman_bfp}
\begin{array}{l}
 \mathcal{S}_{\eps}^a = \Big\{(a,-ma,a)+r_L(1,-m-\lambda_L,-\eps/\lambda_L);\;r_L\leq \rho-a\Big\}, \\
 \mathcal{S}_{\eps}^r = \Big\{(a,ka,a)+r_R(1,k-\lambda_R,-\eps/\lambda_R);\;r_R\geq \mu-a\Big\},
 \end{array}
\end{equation}
respectively. In the next lemma, we establish conditions on the parameters to guarantee the connection between both slow manifolds. These conditions are given by power series of $\varepsilon$ obtained using the method of the undetermined coefficients.

\begin{lemma}\label{lema:m.k.bfpoint}
Consider system \eqref{eq:odeBuffer}, with $f$ given by \eqref{eq:f.3pw}, and $\varepsilon>0$ small enough. Assume that the behavior in the left half space is of saddle-focus type with $\lambda_L\alpha_L<0$, in the right half space is of node-focus type with $\lambda_R\alpha_R>0$ and the behavior in the central strip is of saddle-center type, i.e. $\alpha_C=0$. Setting $\rho=-k\varepsilon -\varepsilon^2$ and $\mu=m\varepsilon-\varepsilon^2$, there exist positive functions $m(a,\varepsilon)$ and $k(a,\varepsilon)$  such that, when $m=m(a,\varepsilon)$ and $k=k(a,\varepsilon)$, the slow manifolds $\mathcal{S}_{\eps}^a$ and $\mathcal{S}_{\eps}^r$ connect. Moreover, $m(a,\varepsilon)$ and $k(a,\varepsilon)$ can be expressed in series expansion of $\varepsilon$ as
\begin{align*}
m(a,\varepsilon) &= \frac {a\pi}{2(1-a)} +O(\varepsilon),\\
k(a,\varepsilon) &= \frac {a\pi}{2(1-a)} +O(\varepsilon). 
\end{align*}
\end{lemma}

\begin{proof}
We first ensure that all the conditions stated in the hypotheses are compatible, in particular, with the continuity of the system, i.e., that relation \eqref{eq:l.n} is satisfied. 

The characteristic polynomial of the system \eqref{eq:centralsys.bfpoint} is $p_C(\lambda)=-\lambda^3+l\lambda^2-\lambda-\epsilon$. Since we assume a configuration of saddle-center type, i.e. $p_C(\lambda)= -(\lambda-\lambda_C)(\lambda^2+\beta_C^2),
$ by equalling coefficients, we get $l=\lambda_C=-\epsilon$ and $\beta_C=1$. Hence equation  \eqref{eq:l.n} writes 
\begin{equation}\label{eq:l.-eps}
-\epsilon = \frac{m\rho+k\mu}{\mu-\rho},
\end{equation}
and from the expressions of $\mu$ and $\rho$ given in the statement of the lemma, we conclude that equation \eqref{eq:l.n} is satisfied.  

Proceeding similarly with the lateral systems, we can relate the eigenvalues in terms of the parameters by equalling the coefficients of the characteristic polynomial.  From this, we get the following sets of relations
$$
\left\{
\begin{array}{ccl}
-m & = & 2\alpha_L+\lambda_L, \\
1 & = & 2\alpha_L\lambda_L+(\alpha_L^2+\beta_L^2), \\
-\epsilon & = & \lambda_L(\alpha_L^2+\beta_L^2),
\end{array}
\right.
\quad \text{and}\quad
\left\{
\begin{array}{ccl}
k & = & 2\alpha_R+\lambda_R, \\
1 & = & 2\alpha_R\lambda_R+(\alpha_R^2+\beta_R^2), \\
-\epsilon & = & \lambda_R(\alpha_R^2+\beta_R^2).
\end{array}
\right.
$$
We note that these conditions are compatible with $m>0$,  $\alpha_L<0$, $\lambda_L<0$, and $k>0$, $\alpha_R>0$ and $\lambda_R<0$.

From \eqref{def:slowman_bfp}, the intersection points $\mathbf{p}_{\varepsilon}^a$ and $\mathbf{p}_{\varepsilon}^r$, of the attracting and repelling slow manifold with the planes $\{x=\rho\}$ and $\{x=\mu\}$ are 
\begin{equation}
    \begin{array}{l}
    \mathbf{p}_{\varepsilon}^a=\left(\rho,-ma-(\rho-a)(m+\lambda_L),a-(\rho-a)\frac{\epsilon}{\lambda_L}\right)\\ \\
    \mathbf{p}_{\varepsilon}^r=\left(\mu,ka + (\mu-a)(k-\lambda_R),a - (\mu-a)\frac{\epsilon}{\lambda_R}\right),\\    
   \end{array}
\end{equation}
respectively. Therefore, the condition for the connection between both slow manifolds can be written as $\mathbf{u}_C(\tau;0,\mathbf{p}_{\varepsilon}^a)=\mathbf{p}_{\varepsilon}^r$, where $\mathbf{u}_C$ is given in \eqref{eq:uC}. Expressing the solution $\mathbf{u}_C$ in terms of the coordinates, the connection condition becomes in a system of three equations and five unknowns $(a,m,k,\tau,\varepsilon)$,
\begin{equation} \label{eq:5eq}
\begin{array}{rcl}
x_C(\tau;0,\mathbf{p}_{\varepsilon}^a) & = & \mu, \\
y_C(\tau;0,\mathbf{p}_{\varepsilon}^a) & = & ka + (\mu-a)(k-\lambda_R), \\ 
z_C(\tau;0,\mathbf{p}_{\varepsilon}^a) & = & a - (\mu-a)\frac{\epsilon}{\lambda_R},
\end{array}
\end{equation}
which is equivalent to the system obtained by dividing by $\varepsilon$ every equation,
\begin{equation} \label{eq:3eq}
\begin{array}{rcl}
\frac{1}{\varepsilon}(x_C(\tau;0,\mathbf{p}_{\varepsilon}^a)-\mu) & = & 0, \\
\frac{1}{\varepsilon}(y_C(\tau;0,\mathbf{p}_{\varepsilon}^a) - ka - (\mu-a)(k-\lambda_R)) & = & 0, \\ 
\frac{1}{\varepsilon}(z_C(\tau;0,\mathbf{p}_{\varepsilon}^a) - a + (\mu-a)\frac{\epsilon}{\lambda_R}) & = & 0.
\end{array}
\end{equation}

Following the expression of $x_C,y_C$ and $z_C$ in \eqref{eq:uC} it can be concluded that for every $a\not \in \{0,1\}$, the vector $\mathbf{s}=(a, \frac{a\pi}{2(1-a)}, \frac{a\pi}{2(1-a)}, \pi, 0)$ is a solution of system \eqref{eq:3eq}. Moreover, computing the derivative of the system with respect to the variables $m,k,$ and $\tau$ evaluated at the previous solution we obtain that 

$$
D_{m,k,\tau} (\mathbf{s})=
\begin{pmatrix}
2a-1 & 1 & a \\
 0 & 0 & a\frac{a\pi}{2(a-1)} \\
 a-1 & a-1 & a
\end{pmatrix},
$$ \sloppy
having determinant $\det(D_{m,k,\tau} (\mathbf{s})) = -a^2\pi(a-1)\neq 0$. From the Implicit Function Theorem we conclude the existence of functions $m(a,\varepsilon)$, $k(a,\varepsilon)$ and $\tau(a,\varepsilon)$, analytic as functions of $\varepsilon$, and such that the parameters $(a,m(a,\varepsilon),k(a,\varepsilon),\tau(a,\varepsilon),\varepsilon)$ are solutions of the system \eqref{eq:3eq}.
\end{proof}

\begin{theorem}\label{th.delay.buf.point}
Under the conditions to allow the connection of the slow manifolds given in~\cref{lema:m.k.bfpoint}, the maximal delay for system \eqref{eq:odeBuffer} is attained at the buffer point, i.e.  $z_d=a-\mu$.
\end{theorem}
\begin{proof}
Under the conditions of~\cref{lema:m.k.bfpoint} both slow manifolds connect and the equilibrium point at $(a,f(a),a)$ is the limit of the solution on the repelling side. Therefore, the $z$ coordinate of the exit point of any orbit coincides with that of the buffer point and then, the distance from the right switching line to the exit point is $a-\mu$.
\end{proof}
\cref{th.delay.buf.point} describes the behavior of the delay in the presence of a buffer point. In~\cref{fig:in.out.buffer}, we show this behavior through the plot of the way-in/way-out function. This plot has a theoretical maximum value at $a-\mu$. However, for the simulations, we consider a truncated expression of the series appearing in \cref{lema:m.k.bfpoint}, and consequently the slow manifolds do not exactly connect and the coordinate of the exit point from the tubular neighbourhood tends to a value lower than the theoretical one. The curve in~\cref{fig:in.out.buffer} was obtained by computing 4 terms in the series expansion of $m(a,\varepsilon)$ and $k(a,\varepsilon)$. If we add new terms in the series expansion then we will get the asymptotic value to be closer to this theoretical value. This is in contrast to~\cref{fig:in-out}, where the way-in/way-out function grows indefinitely as we improve the numerical precision. 

\begin{figure}[h]
    \centering
    \includegraphics[scale=0.41]{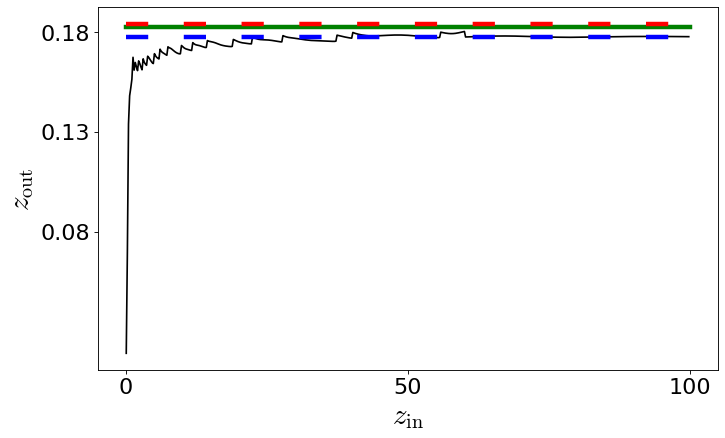}
    \caption{\textbf{Way-in/way-out function in the presence of a buffer point}, defined by considering a tubular neighbourhood of radius $1$ for $a=0.2$, $\epsilon=0.05$, $k\approx 0.3915$ and $m\approx0.3990$, which are given by the first four terms of the series expansion in~\cref{lema:m.k.bfpoint}. The green line corresponds to the distance from the buffer point to the boundary at $x=\mu$, i.e. $a-\mu$; the dashed red line corresponds to the distance from the exit coordinate of the repelling slow manifold to the boundary; the dashed blue line corresponds to distance from the  exit coordinate of the attracting slow manifold to the boundary.}    \label{fig:in.out.buffer}
\end{figure}

Another illustration of the lack of precision in the connection between the slow manifolds is provided in~\cref{fig:SP3d_Buf}. Under the connection conditions given by~\cref{lema:m.k.bfpoint} and satisfied only up to the fourth term, the orbit through the attracting slow manifold leaves the repelling slow manifold instead of approaching the equilibrium. 

As mentioned above, this lack of precision is the cause of the discrepancy between~\cref{th.delay.buf.point} and the results presented in~\cref{fig:in.out.buffer}. What is more, this effect is amplified because of the stiffness of the problem.
\begin{figure}[h]
    \centering
    \includegraphics[scale=0.35]{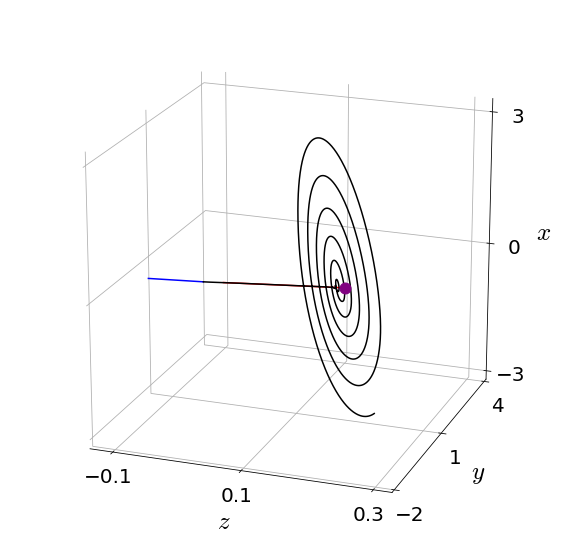}
    \caption{\textbf{Solution of the system with a buffer point.} Plot of the solution (in black), the attracting (blue) and repelling (red) slow manifolds and the buffer point (purple dot), with $a=0.2$, $\eps=0.05$ and the other parameters $k\approx 0.3915$ and $m\approx0.3990$ are fixed for the connection. We take an initial condition  on the attracting slow manifold.}
    \label{fig:SP3d_Buf}
\end{figure}

\section{Application to the Doi-Kumagai neural burster} \label{sec:4dk}
The Doi-Kumagai (DK) model \cite{DKModel} is a very simplified neural burster obtained by the piecewise linearization of a 3-dimensional extension of the Bonhoefer-van der Pol oscillator; see  \cite{HMS85} for details. Its equations read

\begin{equation}\label{eq:DK}
\begin{split}
    \dot{x} & = f(x)-y-z+I, \\
    \dot{y} & = \eta(x-ay), \\
    \dot{z} & = \epsilon(x-bz),
\end{split}
\end{equation}
where
$$
f(x) = -x + \lvert x+1\rvert-\lvert x-1\rvert,
$$
and $\eps$ is a small parameter, i.e., $0<\epsilon\ll 1$. 

As shown in \cite{DKModel,DK05}, the DK model exhibits complex oscillations of elliptic bursting type. This periodic behavior can be understood through the slow dynamics of variable $z$, which drives the fast variables $(x,y)$  between quiescence (resting states of the fast subsystem) and burst regimes (spiking states of the fast subsystem). In particular, in the DK model the slow variable $z$ drives the system from quiescence to burst through a delayed subcritical Hopf bifurcation, and it organises the reverse transition through a delayed fold bifurcation of limit cycles. 

\cref{fig:DK3D}(a) displays a bursting periodic orbit of the DK model, superimposed onto the bifurcation diagram of its fast subsystem. Therefore we also show the critical manifold as family of equilibria of the fast subsystem, together with the subcritical Hopf bifurcation point at the beginning of the oscillatory regime. The fast subsystem possesses two families of limit cycles, which we represent as transparent surfaces: the invariant cone (red) formed by the unstable limit cycles born at the subcritical Hopf bifurcation, as well as the invariant cylinder (blue) of stable limit cycles born through a restabilizing fold of cycles bifurcation, which ends the burst phase in the full system.

The bursting phenomenon in the DK model is obtained by setting parameters in such a way that the system presents a unique equilibrium point. This equilibrium point is an attracting node-focus when it is in the lateral regions, and it becomes a saddle-focus when it is in the central region, which is the case for the bursting cycle presented in~\cref{fig:DK3D}. 

From \cite{DKModel} we consider the bursting regime with parameter values $a = 0.8$, $\eta = 0.5$ and $b = 0.5$. Parameter $I$ is chosen such that $-2.25<I<2.25$, in order for the DK model to have two virtual equilibrium points in the lateral regions and a real equilibrium point $\mathbf{e}=-\frac I{1-1/a-1/b}(1,\frac 1 a,\frac 1 b)$ in the central one.

\begin{figure}[h]
    \centering
    \begin{tabular}{cc}
    \includegraphics[scale=0.32]{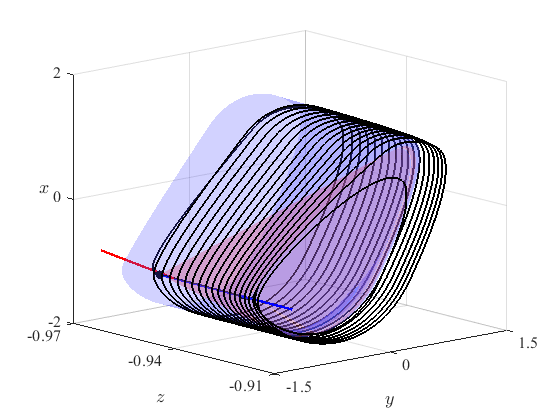}  & \includegraphics[scale=0.18]{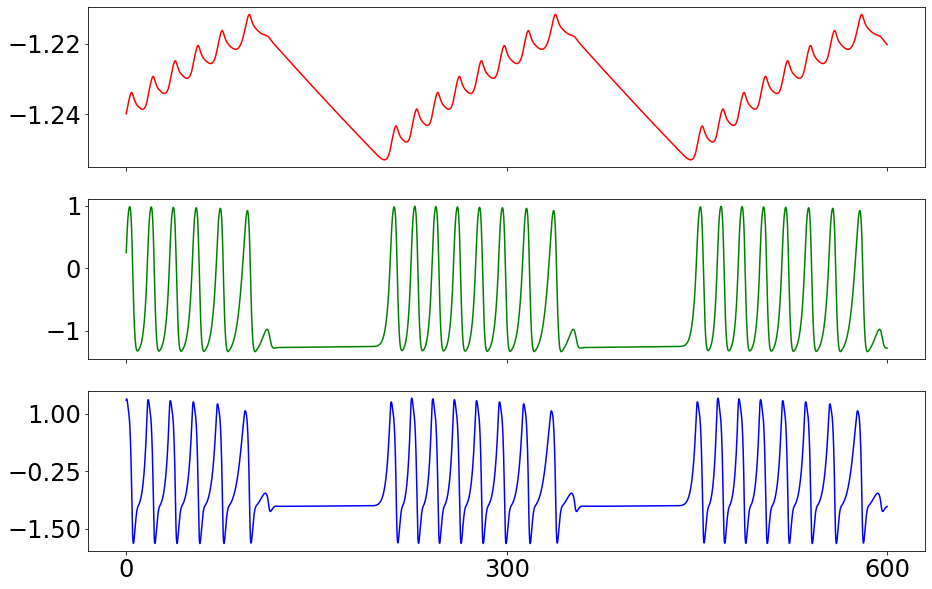} \\
    (a) & (b) 
    \end{tabular}
    
    \caption{\textbf{Eliptic bursting in the DK model}. (a) A 3 dimensional view of a bursting  cycle in the DK model. We can see the canonical slow manifolds (the repelling slow manifold in red and the attracting one in blue), the unstable cone (in red) and the stable cylinder (in blue) of fast subsystem limit cycles, together with a bursting orbit. (b) Time-series of the solution plotted in (a); $x(t)$ in blue, $y(t)$ in green and $z(t)$ in red. We take $I=-1.5$ and $\eps=10^{-3}$.}
    \label{fig:DK3D}
\end{figure}

The critical manifold of the DK model is the polygonal curve 
\[
\mathcal{S}_0=\Big\{(x,y,z): f(x)-y-z+I=0,\;x-ay=0\Big\}.
\]
It is normally hyperbolic everywhere except at the points $(-1,-\frac 1{a},-1+\frac 1{a}+I)$ and $(1,\frac 1{a},1-\frac 1{a}+I)$, located at the intersection between $\mathcal{S}_0$ and the switching planes $\{x=\pm1\}$, at which the equilibrium points change stability.

Owing to the extension of  Fenichel Theory to PWL slow-fast systems developed in \cite{ProhensVich}, the attracting and repelling branches of the critical manifold will perturb for $\eps>0$ small enough, to attracting and repelling slow manifolds, respectively. In~\cref{fig:DK3D}(a) we show the canonical attracting slow manifold (in blue), contained in the half-space $L=\{x<-1\}$ (left region), as well as the canonical repelling slow manifold (in red), contained in the strip $M=\{-1<x<1\}$ (middle region). In the following result we  give conditions for the slow manifolds to connect. 

\begin{lemma} \label{lema:dk.connection}
Consider the DK system~\eqref{eq:DK}. For $\eps$ small enough, the attracting slow manifold (in the half-space $L$) and the repelling slow manifold (in the middle strip $M$) connect if, and only if, the equilibrium point $\mathbf{e}$ is located on the boundary $\{x=-1\}$.
\end{lemma}

\begin{proof}
For $\eps$ small enough, we can consider $\eta a\neq \eps b$. In this case, the attracting and the repelling slow manifolds are given by the line segments parameterised by $\tau$
\begin{equation} \label{eq:u_J}
\mathbf{u}_j(\tau) = \Big\{\mathbf{u}_{0,j} + \tau \mathbf{v}_j \Big\}, \quad \mathbf{v}_j=\left( (\lambda_j+\eta a)(\lambda_j +\eps b), \eta(\lambda_j+\eps b), \eps(\lambda_j + \eta a)\right)^T
\end{equation}
where $j\in\{L,M\}$, $\mathbf{u}_{0,j}=(x_{0,j},y_{0,j},z_{0,j})$ is the equilibrium point in the $j$ region, $\lambda_j$ is the slow eigenvalue in the $j$ region and $\mathbf{v}_j$ the eigenvector associated  with $\lambda_j$. Parameter $\tau$  then satisfies: $x_L(\tau)=x_{0,L} + \tau(\lambda_L+\eta a)(\lambda_L +\eps b)<-1$ and  $|x_M(\tau)|<1$; 
see~\cite{ProhensVich} for details. 

Suppose that the slow manifolds connect on the boundary $x=-1$. Therefore, there exist $\tau_L$ and $\tau_M$ such that (i) $\mathbf{u}_L(\tau_L)=\mathbf{u}_M(\tau_M)$ and (ii) their first coordinates $x_L(\tau_L)=x_M(\tau_M)=-1$. Hence, from  (ii) we obtain that
$$
\tau_j = -\frac{1+x_{0,j}}{(\lambda_j+\eta a)(\lambda_j+\eps b)},\quad j\in{L,M}.
$$

\begin{figure}[!t]
\begin{tabular}{cc}
\includegraphics[scale=0.235]{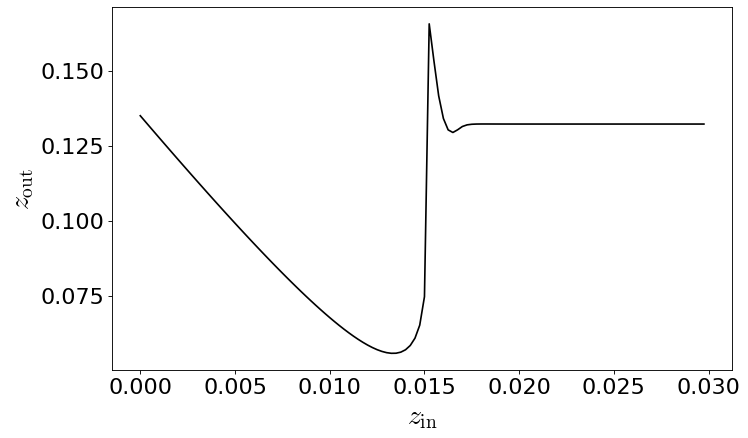} & \includegraphics[scale=0.235]{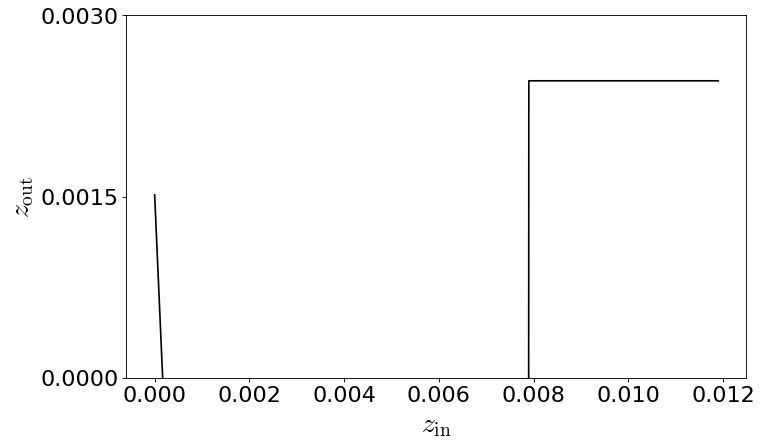}  \\
(a) & (b)
\end{tabular}
\caption{\textbf{Way-in/way-out function and maximal delay for the DK model}, computed for the parameter values $a = 0.8$, $\eta = 0.5$, $b = 0.5$, $I=2$, and; (a) $\varepsilon=10^{-3}$; (b) $\varepsilon=10^{-5}$. After an initial decrease, the function increases and stabilises around the value of the maximal delay. Observe that, as $\varepsilon$ is divided by $10^{2}$, so the maximal delay is divided by a similar amount.}
\label{fig:DK3_in.out}
\end{figure} 
 From (ii), equalling $y_L(\tau_L)=y_M(\tau_M)$ and $z_L(\tau_L)=z_M(\tau_M)$ in expression~\eqref{eq:u_J} and considering also the $\tau_j$ expressions obtained, we obtain the following linear system depending on the $x$-coordinates of the equilibrium points,
\begin{equation} \label{eq:systemX0LX0M}
\begin{array}{l}
    \displaystyle \frac{x_{0,L}}{a} - \frac{(1+x_{0,L})\eta}{(\lambda_L+\eta a)} = \frac{x_{0,M}}{a} - \frac{(1+x_{0,M})\eta}{(\lambda_M+\eta a)}, \\
    \displaystyle \frac{x_{0,L}}{b} - \frac{(1+x_{0,L})\epsilon}{(\lambda_L+\epsilon b)} = \frac{x_{0,M}}{b} - \frac{(1+x_{0,M})\epsilon}{(\lambda_M+\epsilon b)}. \end{array}
\end{equation}
Notice that the determinant of the matrix associated with each system is given by 
$$
\det (\Lambda) = \frac{-\lambda_L\lambda_M(\lambda_M-\lambda_L)(\epsilon b-\eta a)}{ab(\lambda_L+\eta a)(\lambda_M+\eta a)(\lambda_L+\epsilon b)(\lambda_M+\epsilon b)}.
$$
In order to see if system~\eqref{eq:systemX0LX0M} has a unique solution, we have to determine whether the previous determinant is zero or not. To do so, we consider the characteristic polynomials of system~\eqref{eq:DK} in each region, which are given by 
$$
p_j(\lambda)=-\lambda^3 + (f'-\epsilon b-\eta a)\lambda^2 -(\eta+\epsilon -\epsilon bf'+\epsilon b\eta a - \eta af')\lambda - \eta \epsilon( b+a - ba f')
$$
where $f'$ stands for the derivative of $f(x)$ in the corresponding region $j\in\{L,M\}$. Since the determinant is zero when $\lambda_L=\lambda_M$, this leads to $p_L(\lambda_L)-p_M(\lambda_L)=0$. Thus,
$$
\lambda_L^2 + (\epsilon b+ \eta a)\lambda_L + \epsilon b\eta a = 0,
$$
which implies that either $\lambda_L=-\eps b$ or $\lambda_L=-\eta a$. In any case, $p_j(-\eta a)=\eta (\eta a-\eps b)$, $p_j(-\eps b)=\eps(\eps b- \eta a)$ and both imply that $\eps b=\eta a$. Therefore, $\det(\Lambda)\neq 0$ and so the system has a unique solution at $x_{0,L}=x_{0,M}=-1$. In particular, this solution is the equilibrium point of the differential equation located in between the half-space $L$ and the middle strip $M$.
\end{proof}

Following~\cref{lema:dk.connection}, when the equilibrium point is not on the boundary, the attracting and the repelling slow manifolds do not connect. Locally, this configuration is similar to those analyzed in~\cref{sec.2regimes}, \cref{thm:du}. We concluded that in this no-connection scenario, the maximal delay decreases to zero as $\varepsilon$ tends to zero. This behavior of the maximal delay can hence be observed in the DK model; see~\cref{fig:DK3_in.out} where the way-in/way-out function is plotted for two different values of $\varepsilon$. In panel (a) we take $\varepsilon=10^{-3}$, whereas in panel (b) $\varepsilon=10^{-5}$.  In the latter case we observe that the maximal delay is also divided by approximately $10^{2}$, which suggests a  dependence on $\varepsilon$. \cref{thm:du} also provides, in the minimal model, the dependence of the maximal delay on $\varepsilon$, for $\varepsilon$ small enough. In particular, the maximal delay decreases like $u_1\varepsilon+u_2\varepsilon \ln(\varepsilon)$. In~\cref{fig:maxdelay.v.eps}, we represent the maximal delay in the DK model as a function of $\eps$ for two values of parameter $I$. For suitable values of the constants $u_1$ and $u_2$ we also trace the curve $u_1\varepsilon+u_2\varepsilon \ln(\varepsilon)$ which shows that it fits well the predicted behavior of the maximal delay. On the other hand, as $\varepsilon$ increases, the maximal delay tends to the value $1$ or $1/2$, which corresponds with the coordinate of the equilibrium point for the chosen $I$ value, that is, when $I=2$ or $I=0$, respectively. 

\begin{figure}[h]
\centering
 \includegraphics[scale=0.5]{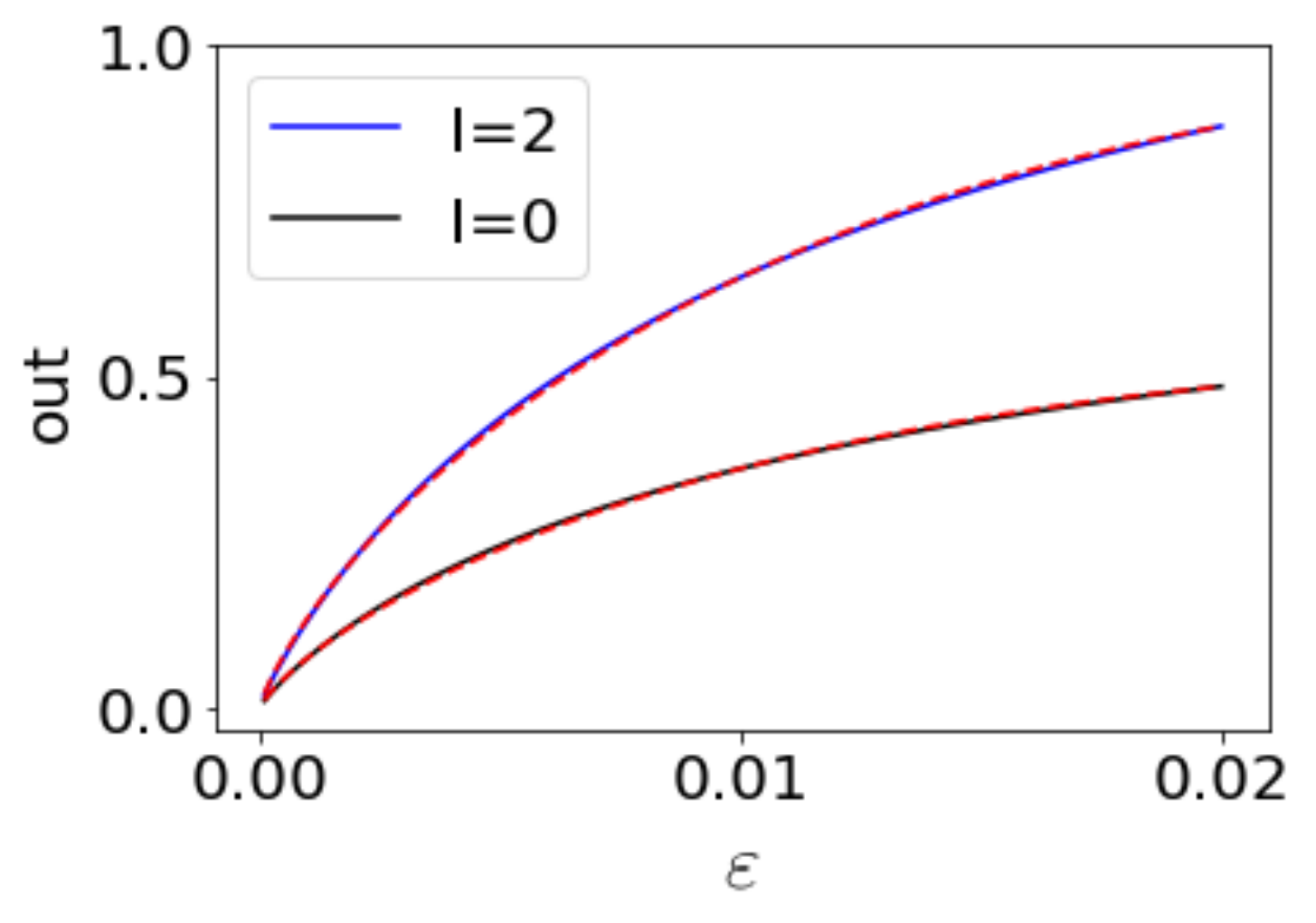}
 \caption{\textbf{Maximal delay as a function of $\eps$ in the DK model}, computed over an interval of $\eps$-values, for two different values of parameter $I$. These values are obtained by setting initial conditions on the attracting slow manifold at $\{x=-1\}$. The remaining parameter values are $a = 0.8$, $\eta = 0.5$ and $b = 0.5$, and the radius of the  tubular neighbourhood is equal to $1$. For suitable values of the constants $u_1$ and $u_2$, we also plot the curve $u_1 \varepsilon+u_2 \varepsilon \ln(\varepsilon)$ (dashed), which fits well with the computed maximal delay curves. }\label{fig:maxdelay.v.eps}
\end{figure}
Next, we slightly modify the DK model  in order to allow a connection between the attracting slow manifold (contained in the half-space L) and the repelling one (in the middle strip M). That is, we add an extra linearity zone in between these two regions. Since we maintain the equilibrium in the middle strip $\{|x|\leq 1\}$, the expected behavior of the slow passage is the one described in~\cref{sec:buff.point} about the presence of a buffer point, in particular in~\cref{th.delay.buf.point}.
Therefore, let us consider system~\eqref{eq:DK}, where we rewrite $f$ as
$$
f(x) = -x-2+(1+s)\frac{\lvert x-\rho\rvert+(x-\rho)}{2} + (1-s)\frac{\lvert x-\mu\rvert+(x-\mu)}{2} -\lvert x-1\rvert-(x-1).
$$
This extended version of the DK model has 3 new parameters, namely $\rho$ and $\mu$, which are the left and the right boundary of the new region, respectively, and parameter $s$, which is the slope of the $x$-nullcline in this new region. Considering $\mu = \frac{\rho(s+1)+2}{s-1}$, we set one of these parameters in such a way that, for $x\not\in(\rho,\mu)$, the vector field is exactly the same as the classical DK model.

The slow manifolds in the modified DK are given by equation \eqref{eq:u_J}. Hence, the intersection points with these manifolds are given by
\begin{align*}
& \tilde{\mathbf{p}}^a = \mathbf{u}_{0,L} + \rho((\lambda_L+\eta a)(\lambda_L +\eps b),\eta(\lambda_L+\eps b), \eps(\lambda_L+\eta a)), \\
& \tilde{\mathbf{p}}^r = \mathbf{u}_{0,R} + \mu((\lambda_R+\eta a)(\lambda_R +\eps b),\eta(\lambda_R+\eps b), \eps(\lambda_R+\eta a)).
\end{align*}

In order to guarantee the connection between the two slow manifolds, we solve the following system
\[\varphi(t;0,\tilde{\mathbf{p}}^a) = \tilde{\mathbf{p}}^r \]
by applying the shooting method considering $\eps,\rho$ and $s$ as unknowns. This procedure allows us to obtain the aforementioned connection up to a prescribed error.

\begin{figure}[h]
    \centering
    \begin{tabular}{cc}
    \includegraphics[scale=0.33]{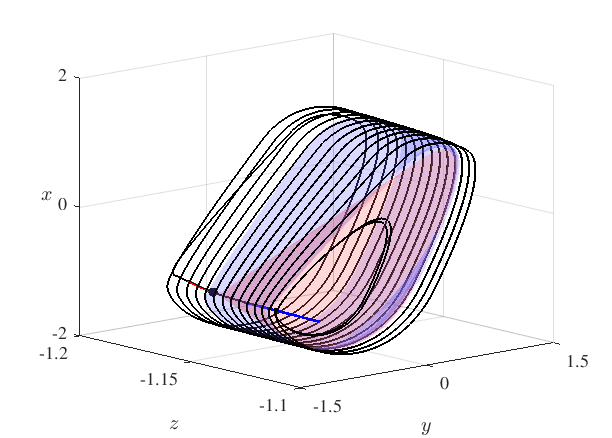}  & \includegraphics[scale=0.17]{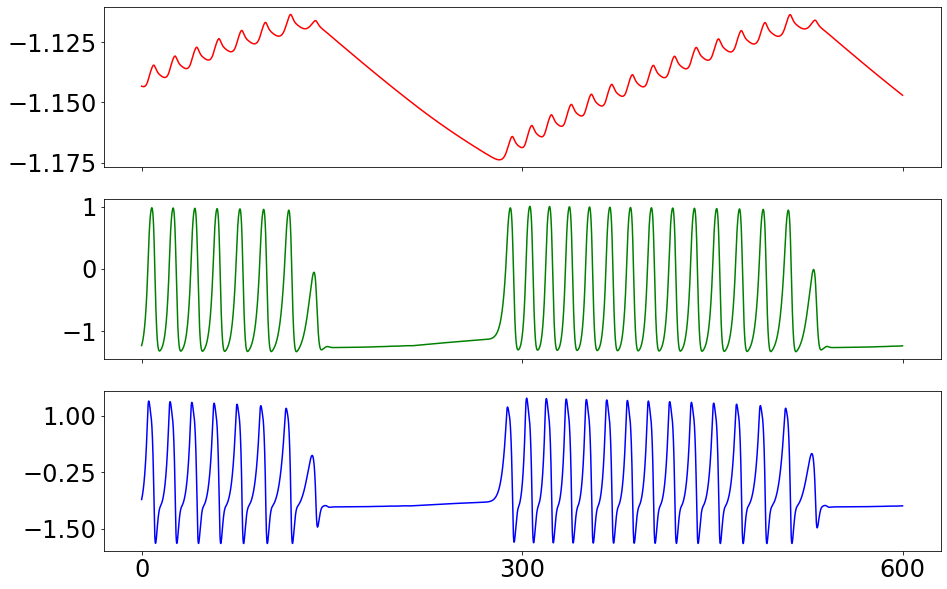} \\
    (a) & (b) 
    \end{tabular}
    
    \caption{\textbf{Eliptic bursting in the modified DK model}. (a) 3 dimensional view of the bursting cycle in the extended DK model, with a fourth region between the L and M regions.  Also shown are the invariant branches of the critical manifold (the repelling branch in red and the attracting one in blue), the unstable cone (in red) and the stable cylinder (in blue) of fast subsystem limit cycles. (b) Time-series of the solution plotted in (a); $x(t)$ in blue, $y(t)$ in green and $z(t)$ in red. We take $I=-1.4$ and $\eps=10^{-3}$.}
    \label{fig:DK3D2}
\end{figure}

Under suitable conditions for the connection to occur,~\cref{fig:DK3D2} shows a bursting cycle in the modified DK model, whose way-in/way-out function takes the form depicted in~\cref{fig:DK4_in.out}. Following~\cref{th.delay.buf.point}, the way-in/way-out function has to reach the same plateau value for every value of $\varepsilon$, this plateau value being  related to the position of the equilibrium point. However, after comparing panels (a) and (b) in~\cref{fig:DK4_in.out}, the maximal delay seems to decrease when $\eps$ gets smaller, hinting at the fact that the connection has not occurred. This discrepancy is not only related to the tolerance of the shooting method we use to approach the connection conditions, but also affected by the strong repulsion of the repelling slow manifold. 

To end this section, we study the latter effect through the analysis of the ratio between the real eigenvalue $\lambda_M$ and the real part $\alpha_M$ of the complex one, which organises the linear dynamics of the system in the strip $\{\mu <x< 1\}$. We have

\begin{figure}[h]
\begin{tabular}{cc}
\includegraphics[scale=0.235]{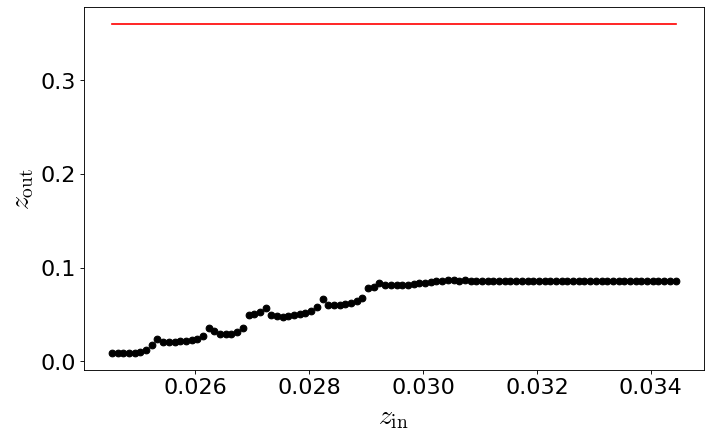} & \includegraphics[scale=0.235]{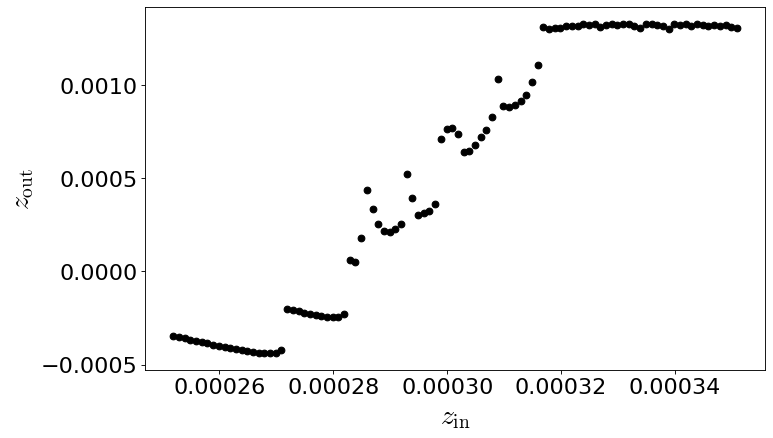}  \\
(a) & (b)
\end{tabular}
\caption{\textbf{Way-in/way-out function in the modified DK model.} Fixed parameter values are $a = 0.8$, $\eta = 0.5$, $b = 0.5$, and $I=-1.4$ (a) $\varepsilon=10^{-3}$, $s=0.4552$ and $\rho=-1.0067$, with the radius of the tubular neighbourhood as $\delta=10^{-2}$, (b) $\varepsilon=10^{-5}$, $s=0.461129$ and $\rho=-1.000068$, with $\delta=10^{-2}$. The paramaters $s$ and $\rho$ are chosen to guarantee a connection between the attracting and the repelling slow manifolds. The red line in panel (a) corresponds to the buffer point value, which is not shown in (b) for a clearer representation. }
\label{fig:DK4_in.out}
\end{figure}

\begin{align*}
\lambda_M = \frac{1}{a+1}(ab-a-b)\varepsilon + O(\varepsilon^2),\\
\alpha_M = \frac{1-a\eta}{2} + \frac{a}{2(a+1)}(1-2b)\varepsilon + O(\varepsilon^2).
\end{align*}

As can be seen, $\lambda_M$ is of order 1 in $\varepsilon$, whereas $\alpha_M$ is of order 0 in $\varepsilon$. As a result, orbits escape very fast from the repelling slow manifold, increasing the stiffness of the problem as $\varepsilon$ tends to zero. In order to avoid this effect, we set parameter $\eta$ such that 
\begin{equation}\label{eq:Mod_DK_nonstiff}
\eta=1/a + \eta_1 \varepsilon \quad \text{with}\quad \eta_1<0,
\end{equation}
which allows for $\lambda_M$ and $\alpha_M$ to be of the same order in $\varepsilon$. In~\cref{fig:in.out.eta.dk}, we show the way-in/way-out function of the modified Doi-Kumagai model, after imposing the relation \eqref{eq:Mod_DK_nonstiff}. As can be observed, the maximal delay behaves in the way described in~\cref{th.delay.buf.point}, which predicts that the maximal delay will be reached at the buffer point, located at $a-\mu$, independently of the chosen $\eps$ value.

\begin{figure}[ht]
    \begin{tabular}{cc}
    \includegraphics[scale=0.235]{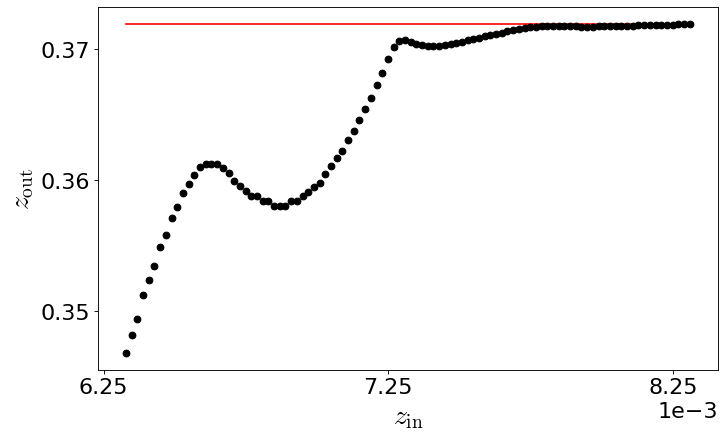}  &
    \includegraphics[scale=0.235]{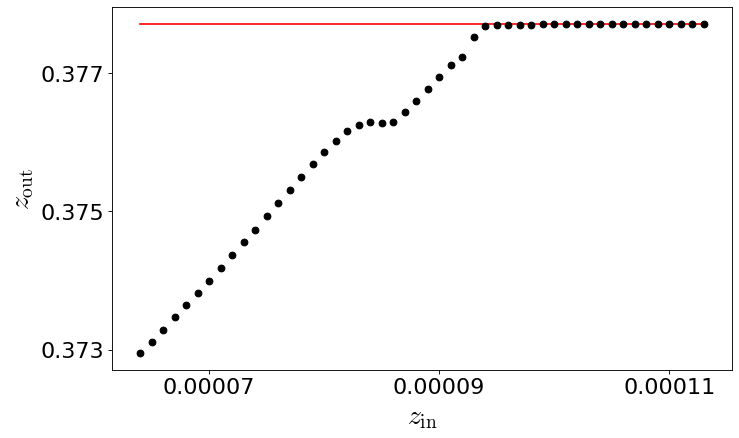}  \\ (a) & (b)
    \end{tabular}
    \caption{\textbf{Way-in/way-out function in the modified DK model, with eigenvalues of the same order.} Fixed parameter values are $\eta=1/a - 10\varepsilon$, $a=0.8$, $b=0.5$ and $I=-1.4$. The red lines correspond to the buffer point. (a) $\varepsilon=10^{-3}$, $s=0.6165$ and $\rho=-1.0018$, with $\delta=10^{-2}$, (b) $\varepsilon=10^{-5}$, $s=0.625649$ and $\rho=1.000018$, with $\delta=10^{-2}$. Parameters $s$ and $\rho$ are chosen to guarantee a connection between the attracting and the repelling slow manifolds.
    } 
    \label{fig:in.out.eta.dk}
\end{figure}

\section{Discussion}
In this paper, we have analyzed the phenomenon of slow passage through a Hopf bifurcation in the context of piecewise linear slow-fast dynamical systems, both qualitatively and quantitatively. Delayed loss of stability is a common behavior naturally present in smooth systems and we have shown that the PWL framework allows to reproduce  this delay. However, different scenarios can arise depending on the minimal model we use, as described in all cases through the way-in/way-out function and the maximal delay's asymptotic value.

We first considered a minimal PWL model exhibiting a delayed loss of stability, in such a way that the maximal delay $z_d$ depends on $\eps$, and tends to zero with $\eps$. We obtained an explicit expression for the dependence of $z_d$ on $\eps$ by taking advantage of the PWL framework. The main ingredient for such a behavior of the maximal delay is the $O(\eps)$ distance between the attracting and the repelling slow manifolds.

We then extended this minimal PWL model by introducing a third linearity region allowing the connection between the canonical slow manifolds. This strategy has proven useful to fully analyse PWL  versions of smooth slow-fast dynamics near a quadratic fold of the critical manifold, that is, recovering also the canard regime; see~\cite{desroches2016,fernandez2016canard}. Hence, the connection of the canonical slow manifolds allowed the maximal delay to tend to infinity. This behavior is more similar to what happens in smooth systems. However, because in the PWL context the divergence along the repelling slow manifold remains constant from the moment it crosses the separation plane, the slow passage phenomenon has a very unstable character. Hence, given that the conditions for a connection to occur are not, in practice, entirely satisfied, and due to round-off errors during the simulations, the computed way-in/way-out function always behaves as if the maximal delay has a finite value that cannot be exceeded, similar to a buffer point. 

In order to study the presence of a proper buffer point, we also considered the case of a three-region system having an equilibrium point in the central region. Results show the appearance of a boundary for the solutions, which does not depend on the previously observed drawbacks, revealing the existence of the buffer point.

Since the slow passage phenomenon is a key element to generate bursting dynamics, we checked the validity of our results on a PWL elliptic bursting model, namely, the DK model. This adds to the existing literature on bursting dynamics in the PWL context~\cite{deng2009conceptual,desroches2016canards}. However, the slow passage phenomenon in this model appears to be too weak since, as in~\cref{sec.2regimes}, the distance between the attracting and the repelling slow manifold is linear with $\eps$. To make it stronger, we added to the DK model a new linearity region allowing for the connection between both the attracting and the repelling slow manifolds. However, even though we could connect them, the problem appeared to be so unstable that we could not observe any significant difference with the case where the slow manifolds do not connect. Thus, even when theoretically the maximal delay must approach the buffer point independently of $\varepsilon$, in practice, the numerical precision cannot afford the stiffness of the problem and the maximal delay decreases with $\varepsilon$.

To overcome the stiffness problem, we finally studied a suitable parameter set of the modified DK system, not only allowing for the connection, but also providing the real eigenvalue and the real part of the complex eigenvalues of the linear system in $M$ to be of the same order in $\varepsilon$. This reduced the stiffness of the system and maintained the maximal delay high, irrespective the value of $\varepsilon$. However, further efforts would be needed to find suitable parameter sets allowing to exhibit bursting oscillations for this new scenario as well. This is an interesting question for future work.

\section{Acknowledgments}
JP, AET and CV are partially supported by the Ministerio de Ciencia, Innovación y Universidades (MCIU) project PID2020-118726GB-I00. AET and CV are also partially supported by the Ministerio de Economia y Competitividad through the project MTM2017-83568-P (AEI/ERDF,EU). 

\bibliographystyle{plain}
\bibliography{references}

\appendix
\section{Solutions of the systems}\label{app:solutions}
Considering the two-region system \eqref{sys:2pieces}-\eqref{eq:x-nullcina1}, its local solution through the point $\mathbf{p}=(x_0,y_0,z_0)$, which we denote as $\mathbf{u}(t;0,\mathbf{p})$, has the first and second coordinates given by 
\begin{equation}\label{eq:xL}
\begin{split}
x_L(t;0,\mathbf{p})&=z(t;0,\mathbf{p})-m\varepsilon+
  \mathrm{e}^{-\frac {m}{2}t} \Big(
      (x_0-z_0+m\epsilon) \cos\left(\xi_m t \right) \\
      &
      +\frac{-m(z_0-m\epsilon+x_0)-2(y_0+\epsilon)}{\sqrt{4-m^2}} \sin\left(\xi_m t \right)
    \Big),\\
y_L(t;0,\mathbf{p})&=-m(z(t;0,\mathbf{p})-m\varepsilon)-\varepsilon\\
  &
  +\mathrm{e}^{-\frac{m}{2}t} \Big(
      (m(z_0-m\epsilon)+(y_0+\epsilon))\cos\left(\xi_m t \right) \\
      &
      +\frac{m^2(z_0-m\epsilon)+m(y_0+\epsilon)+2(x_0-z_0+m\epsilon)}{\sqrt{4-m^2}} \sin\left(\xi_m t \right)
    \Big),
\end{split}
\end{equation}
if it is contained in the left region, and
\begin{equation}\label{eq:xR}
\begin{split}
x_R(t;0,\mathbf{p})&=z(t;0,\mathbf{p})+k\varepsilon+
  \mathrm{e}^{\frac {k}{2}t} \Big(
      (x_0-z_0-k\epsilon) \cos\left(\xi_k t \right) \\
      &
      +\frac{k(z_0+k\epsilon+x_0)-2(y_0+\epsilon)}{\sqrt{4-k^2}} \sin\left(\xi_k t \right)
    \Big),\\
y_R(t;0,\mathbf{p})&=k(z(t;0,\mathbf{p})+k\varepsilon)-\varepsilon\\
  &
  +\mathrm{e}^{\frac {k}{2}t} \Big(
      (-k(z_0+k\epsilon)+(y_0+\epsilon))\cos\left(\xi_k t \right) \\
      &
      +\frac{k^2(z_0+k\epsilon)-k(y_0+\epsilon)+2(x_0-z_0-k\epsilon)}{\sqrt{4-k^2}} \sin\left(\xi_k t \right)
    \Big),
\end{split}
\end{equation}
if it is contained in the right region, where $\xi_m = \frac {\sqrt{4-m^2}}2$, $\xi_k = \frac {\sqrt{4-k^2}}2$.  The third coordinate remains similar in all regions and it is the slow drift given by $z(t;0,\mathbf{p}) = z_0 + \eps t$.

In the case of the three-region system  \eqref{sys:2pieces}-\eqref{eq:f.3pw}, the local solution through $\mathbf{p}$ is given by \eqref{eq:xL} and \eqref{eq:xR}, if it is contained in the left or right regions and by

\begin{equation}\label{eq:xC}
\begin{split}
x_C(t;0,\mathbf{p}) & =z(t;0,\mathbf{p}) + l\varepsilon +
  \mathrm{e}^{\frac{l}{2}t} \Big(
      (x_0-(z_0+l\epsilon)) \cos\left(\xi_l t \right) \\
       &  +
      \frac{1}{\xi_l}(\frac{l}{2}(x_0+z_0+l\epsilon)+n-(y_0+\epsilon)) \sin\left(\xi_l t \right)
    \Big),\\
y_C(t;0,\mathbf{p})&=l(z(t;0,\mathbf{p})+l\varepsilon)+n-\varepsilon + \mathrm{e}^{\frac {l}{2}t} \Big(
    (y_0-(l(z_0+l\epsilon)+n-\epsilon)) \cos\left(\xi_l t \right)
      \\
    & + \frac{1}{2\xi_l}(l^2(z_0+l\epsilon)+2(x_0-(z_0+l\epsilon))+l(n-(y_0+\epsilon))) \sin\left(\xi_l t \right)
    \Big),
\end{split}
\end{equation}
if it is contained in the central region, where $\xi_l = \frac {\sqrt{4-l^2}}2$. The third coordinate is also $z(t;0,\mathbf{p}) = z_0 + \eps t$.

The local expression of the solution in the central region of the buffer-point system~\eqref{eq:odeBuffer}-\eqref{eq:f.3pw} is given by:
\begin{equation}\label{eq:uC}
\mathbf{u}_C(t;0,\mathbf{p})=P\mathrm{e}^{Jt}P^{-1}\left( \mathbf{p}+A^{-1}\mathbf{a}_0\right)-A^{-1}\mathbf{a}_0,
\end{equation}
where
$$
A = \begin{pmatrix}
-\epsilon & -1 & 0 \\
1 & 0 & -1 \\
-\epsilon & 0 & 0
\end{pmatrix}, \quad
P = \begin{pmatrix}
1 & 1 & 0 \\
0 & -\epsilon & -1 \\
1 & 0 & \epsilon
\end{pmatrix}, \quad
\mathbf{a}_0 = \begin{pmatrix}
n \\ 0 \\ \epsilon
\end{pmatrix},
$$
$$
\mathrm{e}^{Jt} = \begin{pmatrix}
\mathrm{e}^{-\epsilon t} & 0 & 0 \\
0 & \cos(t) & \sin(t) \\
0 & -\sin(t) & \cos(t)
\end{pmatrix}.
$$

\section{Way-in/way-out function} 
In this section we provide the pseudo-code algorithm used to compute the way-in/way-out function given a tubular neighborhood of radius $\delta$. When a two-region system is considered, we use $\mu=0$ while in the three-region systems, we consider $\mu$ being the $x$ boundary between the central region and the right one.

\RestyleAlgo{ruled}
\begin{algorithm}
\caption{Way-in/way-out function}
\label{alg:in-out}
\begin{algorithmic}[1]
\STATE{Define vector of $z$-coordinates $z_1,\ldots,z_N$}
\STATE{Let $\delta$ be the radius of the neighbourhood}
\FOR{each $i=1,\ldots,N$}
    \STATE\label{line3}{Let $\mathbf{p}_i=(z_i-m\varepsilon,\varepsilon-m(z_i-m\varepsilon)+\delta,z_i)$}
    \STATE{Solve the ODE system $\mathbf{u}(t_{i,j};0,\mathbf{p}_i)=(x(t_{i,j}),y(t_{i,j}),z(t_{i,j}))$}
    \STATE{Find the time $t_{i,s}$ such that $x(t_{i,s})=\mu$}
    \STATE{Take $j=s$ and $d_{i,j}=\delta$}
    \WHILE{$d_{i,j}\leq \delta$}
        \STATE{Find the plane $\pi(z_0,t_{i,j})$ generated by the complex eigenvectors}
        \STATE{Find the intersection point $\hat{\mathbf{u}}(t_{i,j})=(\hat{x}(t_{i,j}),\hat{y}(t_{i,j}),\hat{z}(t_{i,j}))= \mathcal{S}_{\varepsilon}^r\cap\pi(z_0,t_{i,j})$}
        \STATE{Find the distance $d_{i,j}=d(\hat{\mathbf{u}}(t_{i,j}),\mathbf{u}(t_{i,j};0,\mathbf{p}_i))$}
        \IF{$d_{i,j}>\delta$}
            \STATE{out: $\hat{z}_i = \hat{z}(t_{i,j})$}
        \ELSE \STATE{$j=j+1$}
        \ENDIF
    \ENDWHILE
\ENDFOR
\RETURN pairs $(z_i,\hat{z}_{i})$
\end{algorithmic}
\end{algorithm}

\end{document}